\documentclass[myaap,preprint]{imsart}

\RequirePackage[OT1]{fontenc}
\RequirePackage{amsthm,amsmath}
\RequirePackage[numbers]{natbib}

\startlocaldefs
\numberwithin{equation}{section}
\theoremstyle{plain}
  
\endlocaldefs

\newcommand{\radius}{\rho} 
\newcommand{\size}{N}  
\newcommand{\msize}{N}

\newcommand{\csize}{N^3}
\newcommand{\ti}{t}

\newcommand\given[1][]{\:#1\vert\:} 

\usepackage{mathrsfs}

\newenvironment{definition*}[1][Definition]{\begin{trivlist}
\item[\hskip \labelsep {\bfseries #1}]}{\end{trivlist}}
 
\newenvironment{remark*}[1][Remark]{\begin{trivlist} 
\item[\hskip \labelsep {\bfseries #1}]}{\end{trivlist}}

\usepackage{graphicx}
\graphicspath{{figs/}}

\usepackage[cmex10]{amsmath}
\usepackage{amssymb}
\usepackage{bm}
\usepackage{commath}

\usepackage{amsthm}

\usepackage{enumitem}

\usepackage{array}

\newtheorem{theorem*}{\textbf{Theorem}}
\newtheorem{proposition*}{\textbf{Proposition}}

\newtheorem{lemma*}{\textbf{Lemma}}
\newtheorem{conjecture*}{\textbf{Conjecture}}

\theoremstyle{definition}
\newtheorem{definition}{\textbf{Definition}}
\newtheorem{remark}{\textbf{Remark}}

\newtheorem{corollary*}{Corollary}

\begin{document}

\begin{frontmatter}
{ 
\title{Evolution and Limiting Configuration of a Long-Range  Schelling-Type Spin System\thanksref{T1}}}
\runtitle{Evolution and Limiting Configuration of a  Spin System}
\thankstext{T1}{An extended abstract of an earlier version of this paper containing a subset of the results and with most proofs omitted appears in the proceedings of ACM  Symposium on Theory of Computing (STOC '18). This work was partially supported by Army Research Office (ARO), award number W911NF-15-1-0253.}  
 
\begin{aug}
\author{\fnms{Hamed} \snm{Omidvar}\ead[label=e1]{homidvar@ucsd.edu}} 
\and
\author{\fnms{Massimo} \snm{Franceschetti}\ead[label=e2]{mfrances@ucsd.edu}} 

\runauthor{H. Omidvar and M. Franceschetti.}

\affiliation{University of California, San Diego} 

\address{
University of California, San Diego\\  
Electrical and Computer Engineering Department, \\
9500 Gilman Drive, La Jolla, CA 92093 \\
\printead{e1}\\
\phantom{E-mail:\ }\printead*{e2}}
\end{aug}

\begin{abstract}
We consider a long-range interacting particle system in which binary particles -- whose initial states are chosen uniformly at random -- are located at the nodes of a flat torus $(\mathbb{Z}/h\mathbb{Z})^2$. Each node of the torus is connected to all the nodes located in an $l_\infty$-ball of radius $w$ in the toroidal space centered at itself and we assume that $h$ is  exponentially larger than $w^2$. Based on the states of the neighboring particles and on the value of a common intolerance threshold $\tau$, every particle is labeled ``stable," or ``unstable.'' Every unstable particle that can become stable by flipping its state  is labeled ``p-stable." Finally, unstable particles that  remained p-stable for  a random, independent and identically distributed waiting time, flip their state and become stable. When the waiting times have an exponential distribution and $\tau \le 1/2$, this model is equivalent to a Schelling model of self-organized segregation in an open system, a zero-temperature Ising model with Glauber dynamics, or an Asynchronous Cellular Automaton (ACA) with extended Moore neighborhoods. 
We first prove a shape theorem for the spreading of the ``affected'' nodes of a given state -- namely nodes on which a particle of a given state would be p-stable. As $w \rightarrow \infty$,  this spreading starts with high probability (w.h.p.)  from any $l_\infty$-ball   in the torus having radius $w/2$ and containing only affected nodes, and continues for a time that is at least  exponential in the cardinalilty of the neighborhood of interaction $N = (2w+1)^2$. Second, we show that when the process reaches a limiting  configuration and no more state changes occur, \  for \  all  ${\tau \in (\tau^*,1-\tau^*) \setminus \{1/2\}}$ where ${\tau^* \approx 0.488}$, w.h.p. any particle is contained in a large ``monochromatic ball'' of cardinality exponential in $N$. When particles are placed on the infinite lattice $\mathbb{Z}^2$ rather than on a flat torus, for the values of $\tau$ mentioned above, after  a sufficiently long evolution time, w.h.p. any particle is contained  in a large monochromatic ball of cardinality exponential in~$N$.
\end{abstract}  
\begin{keyword}[class=MSC]
\kwd[Primary ]{60K35} 
\kwd{82C22}
\kwd{82C43}
\kwd{82B43}
\end{keyword}

\begin{keyword}
\kwd{Interacting Particle System (IPS)}
\kwd{Agent-Based Model}
\kwd{Asynchronous Cellular Automaton}
\kwd{Zero-Temperature Ising Model}
\kwd{Unperturbed Schelling Segregation}
\kwd{Distributed Algorithm}
\kwd{Self-Organized Segregation}
\kwd{Percolation Theory}
\kwd{First Passage Percolation}
\kwd{Exponential Segregation}
\end{keyword}
\end{frontmatter} 

\newpage
\section{Introduction}
\subsection{Background} 
Consider a flat torus $(\mathbb{Z}/h\mathbb{Z})^2$. Connect 
 each node of the torus   to all the nodes located in an $l_\infty$-ball of radius $w$ in the toroidal space centered at itself, and assume that   $h$  is  exponentially larger than $w^2$.
 This gives a   Cayley graph $G_w = C((\mathbb{Z}/h\mathbb{Z})^2,\{-w,  -w+1,\ldots,w\}^2\setminus \{(0,0)\})$. Put a particle at each node of the graph such that its initial binary state is an element of the set $\{\theta, \bar{\theta}\}$  that is chosen independently and uniformly at random. 
The ``neighborhood" of a node, or particle,  is defined as the set  containing the node itself and all of its adjacent nodes in $G_w$. 
All particles have a common intolerance threshold $0< \tau < 1$, indicating the minimum fraction of particles in their same state that must be in their neighborhood to label them ``stable."  A particle that is not stable is  labeled ``unstable."
Furthermore, a particle that is not stable, but can become stable by flipping its state, is labeled potentially stable, or ``p-stable.'' Every time a particle becomes p-stable,  it is assigned an independent and identical clock.
 When the clock rings, the particle's state is flipped if the particle has remained p-stable for the entire clock duration. This flipping action indicates that the particle has moved out of the system and a new particle has occupied its location. This change is then immediately detected by the neighbors   who update their labels accordingly. We are interested in the limiting configuration of this process  when both $h$ and $w$ tend to infinity, and $h$ is exponentially larger than $w^2$. Namely,  we consider the behavior of the system when  $\log h=\omega(w^2)$ as  $w \rightarrow \infty$.  This is of course equivalent to looking at the behavior   for $h \rightarrow \infty$ when $w^2 = o(\log h)$, while also requiring $w \rightarrow \infty$.  This choice ensures that   the rate of growth of the number of particles inside any neighborhood is exponentially smaller than the rate of growth of the number of particles over the entire torus.
As a consequence,  interactions  occur over an unbounded range, but are also sufficiently ``local''   to ensure that a limiting configuration can be studied.


In social sciences and economics, the model we have described has been extensively studied using Poisson clocks and is known  as the Schelling model in an ``open" system \cite{schelling1969models,schelling1971dynamic}. In computation theory, mathematics, physics, complexity theory, theoretical biology and material sciences, it is known as a two-dimensional, two-state Asynchronous Cellular Automaton (ACA) with extended Moore neighborhoods and exponential waiting times \cite{chopard1998cellular}. Related models appeared in epidemiology \cite{hethcote2000mathematics, draief2010epidemics}, economics \cite{jackson2002formation}, engineering and computer sciences \cite{kleinberg2007cascading, easley2010networks}.  
Mathematically, all  of them fall in the general area of interacting particle systems~\cite{liggett2012interacting,liggett2013stochastic}. 
For an intolerance value of $1/2$, the model corresponds to the Ising model with zero temperature, which exhibits spontaneous magnetization as spins align along the direction of the local field~\cite{stauffer2007ising, castellano2009}.

The dynamics of these processes can be roughly divided into two classes. \textit{Glauber dynamics} assume unstable particles to simply flip their state if this makes them stable.  In contrast, \textit{Kawasaki dynamics} assume that pairs of unstable particles swap their locations if this will make both of them stable.  While Glauber dynamics correspond to an ``open'' system where  the number of particles of each state can change over time, the Kawasaki dynamics correspond to a ``closed" system where the number of particles of each state is fixed. 
In this paper, we consider    Glauber dynamics.
Other variants are possible, 
including  having unstable particles swap (or flip) regardless of whether this makes them stable or not, or  to assume that particles have a small probability of acting differently than what the general rule prescribes, 
have multiple intolerance levels,  multiple states, different distributions, and time-varying intolerance  \cite{young2001individual, zhang2004dynamic, zhang2004residential, zhang2011tipping, mobius2000formation, meyer2003immigration,bhakta2014clustering, schulze2005potts,barmpalias2018minority,barmpalias2015tipping}.

\begin{figure}
\begin{center}
{\includegraphics[width=\textwidth]{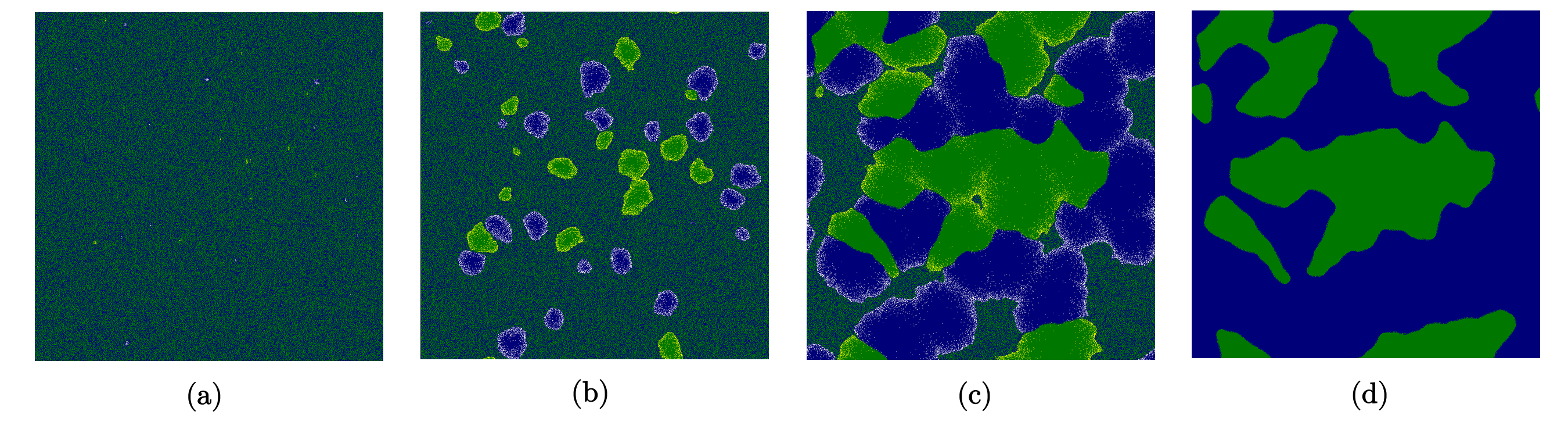}}
\end{center}
\caption{{\footnotesize Self-organization arising over time for a value of the intolerance $\tau=0.42$ on a $1000 \times 1000$ flat torus and neighborhood size $441$. Green and blue indicate areas of ``stable'' particles in states ${\theta}$ and $\bar{\theta}$, respectively. White and yellow indicate areas of ``unstable" particles in states ${\theta}$ and $\bar{\theta}$ respectively.  Initial configuration (a),   transient configurations (b)-(c),  final configuration (d). When the process terminates all particles are stable and large monochromatic areas can be observed. }}
\label{Fig:Sim_results_1}
\end{figure}

A common effect observed by simulating several variants of the model is that 
in the limiting configuration large monochromatic areas containing  particles with the same state are formed, for a wide range of the intolerance threshold $\tau$. This corresponds to observing spontaneous self-organization resulting from local interactions. See Figure~\ref{Fig:Sim_results_1} for a simulation of this behavior.

\subsection{Prior work}

Although simulation results have been available for a long time, rigorous results for the limiting  behavior of the model appeared only recently, even for  the one-dimensional case and assuming Poisson clocks. 
Brandt et al.~\cite{brandt2012analysis}  considered  a  ring graph for the Kawasaki model of evolution.  In this setting, 
they  showed that for an intolerance level $\tau=1/2$, the expected size of the largest monochromatic ball containing an arbitrary particle in the final configuration is polynomial   in the size of the neighborhood. 
Barmpalias et al.~\cite{barmpalias2018digital} showed
that there exists a value of $\tau^*\approx 0.35$, such that for all  $\tau<\tau^*$ the initial configuration remains almost static with high probability (w.h.p.), while for all $\tau^*<\tau<1/2$  the size of the largest monochromatic ball in the final configuration is exponential   in the size of the neighborhood w.h.p. On the other hand, for all $\tau>1/2$  the system evolves w.h.p. towards a  state with only two monochromatic components. For the Glauber model the behavior is similar and consists of a transition from an   almost static configuration to a configuration with exponential monochromatic balls occurring at $\tau \approx 0.35$, a special point $\tau=1/2$ with monochromatic balls of polynomial expected size,  then again exponential monochromatic balls  until $\tau \approx 0.65$, and finally an   almost static configuration for larger values of $\tau$. 
Holden and Sheffield  \cite{holden2020scaling} have considered the case $\tau=1/2$ and studied the dynamical scaling limit as the size of the neighborhood tends to infinity and the lattice is correspondingly re-scaled.


 
In the two-dimensional model, the case $\tau=1/2$ is open. Immorlica et al.~\cite{immorlica2015exponential}  have shown for the Glauber dynamics  the existence of a value $\tau^*< 1/2$, such that for all $\tau^*<\tau<1/2$  the expected size of the largest monochromatic ball is exponential in the size of the neighborhood.  This shows that exponential monochromatic balls are expected  in the small interval $\tau \in (1/2-\epsilon, 1/2)$.
Barmpalias et al.~\cite{barmpalias2016unperturbed} considered a model in which particles in different states have different intolerance parameters, i.e.,  $\tau_1$ and $\tau_2$.  For the special case of $\tau_1 = \tau_2 = \tau$, they have shown that when  $\tau>3/4$, or  $\tau<1/4$,  the initial configuration remains almost static  w.h.p.
 
In a previous work by the authors \cite{omidvar2017self2}, the intolerance interval that leads to the formation of large monochromatic balls has been enlarged from $\epsilon>0$ to $\approx 0.134$, namely  when $0.433 < \tau < 1/2$ (and for $1/2<\tau<0.567$), the expected size of the  largest monochromatic ball   is  exponential  in the size of the neighborhood of interaction. 
In addition, 
``almost monochromatic  balls'' have been considered, namely balls where the ratio of the  number of particles in one state and the  number of particles in the other state quickly vanishes as the size of the neighborhood grows, and it has been shown that for $0.344 < \tau \leq 0.433$ (and for $0.567 \leq \tau<0.656$) the expected size of the largest almost monochromatic ball  is exponential  in the size of the neighborhood. 


\subsection{Contribution} 
The first contribution of this paper is the development of a shape theorem for the spread of ``affected" nodes of a given state $\theta$ -- namely nodes on which $\theta$-particles would be p-stable -- during the process dynamics.
Letting $w$ be the $l_\infty$ radius of the neighborhood of interaction on the torus, and $N=(2w+1)^2$ its cardinality, we show that conditional  on the existence of an $l_{\infty}$ ball   in the torus having radius $w/2$ centered at the origin and containing only $\theta$-affected nodes, 
the spreading  of these affected nodes starts with high probability from  such a ball, and continues for a time that is at least exponential in $N$. 
This is the first result that precisely describes the transient dynamics of the spreading process. Key to this result is that we consider the spreading of  $\theta$-affected nodes rather than the spreading of unstable particles. These nodes can have a particle located on them being p-stable. Thus, they have the potential of hosting a particle  that is unstable and can become stable by flipping its state. While unstable particles may keep switching between being stable and unstable during the process dynamics,  the behavior of $\theta$-affected nodes  is somewhat more static;   once they become affected they remain so for a long time interval or indefinitely. Another key property  is  that in the initial configuration   $\theta$-affected nodes  are rare: they   do not occur with high probability in an exponentially large region around the origin. It follows that all the $\theta$-affected nodes that are found within this region after an exponentially long evolution time  must have  spread  from the original affected region of radius $w/2$ centered at the origin.


 Our second contribution is determining the limiting size  of the largest monochromatic ball, for a given interval of $\tau$. A weakness of all previous results for the two dimensional case is that they obtain lower bounds on the expected size  of the largest monochromatic ball  containing a given particle, but they do not show that in the final configuration any particle ends up  in an exponentially large monochromatic ball with high probability. A possibility that is consistent with the results in the literature (but inconsistent with the simulation results) is that only an exponentially small fraction of the nodes are contained in a large monochromatic ball  at the end of the process, but that those neighborhoods are so large that the expected size  of largest monochromatic ball  containing any node is exponentially large. For this reason, current results leave a large gap in our qualitative understanding of the two-dimensional process. We show that when 
the process stops, \  for \  all  ${\tau \in (\tau^*,1-\tau^*) \setminus \{1/2\}}$ where ${\tau^* \approx 0.488}$,  w.h.p. any particle is  contained in a large monochromatic ball of size exponential in  $N = (2w+1)^2$.
When particles are placed on the infinite lattice $\mathbb{Z}^2$ rather than on a flat torus, for the values of $\tau$ mentioned above, after a sufficiently long  evolution time,  w.h.p. any particle is  contained in a large monochromatic ball of size exponential in $\size$. These results are summarized in Figure~\ref{fig:contribution}.



\begin{figure}[!t] 
\centering
\includegraphics[width=3in]{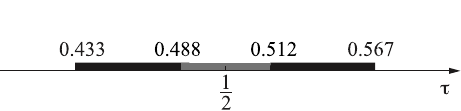}
\caption{ For  ${\tau \in (\tau_*,1-\tau_*) \setminus \{1/2\}}$, where ${\tau_* \approx 0.433}$, we prove a shape theorem for the spread of the ``affected" nodes during the process dynamics  (gray plus black region) and then show that in the final configuration, \  for \  all  ${\tau \in (\tau^*,1-\tau^*) \setminus \{1/2\}}$ where ${\tau^* \approx 0.488}$, with probability tending to one as $w\rightarrow \infty$ (w.h.p.), any particle is contained in a large ``monochromatic ball'' of  size exponential in $\size=(2w+1)^2$ (gray region). } 
\label{fig:contribution}
\end{figure}

\subsection{Additional related work}
For the case of a neighborhood of interaction of constant radius $w$ over an infinite lattice, Fontes et al.~\cite{siv} have shown the existence of a critical probability  $1/2<p^*<1$ for the initial Bernoulli distribution of the particle states such that for $\tau =1/2$  and $p > p^*$ the Glauber model on the $d$-dimensional grid converges to a state where only particles in one state are present. This shows that complete monochromaticity occurs w.h.p.\ for $\tau=1/2$ and $p \in (1-\epsilon,1)$. Morris~\cite{morris2011zero} has shown that $p^*$   converges to $1/2$ as $d \rightarrow \infty$. Caputo and Martinelli \cite{caputo2006phase} have shown the same result for $d$-regular trees, while  Kanoria and Montanari~\cite{montanaritree} derived it for  $d$-regular trees in a synchronous setting where  flips occur simultaneously,  and obtained lower bounds on $p^*(d)$ for small values of $d$. The case $d = 1$   was first investigated by Erd\"{o}s and Ney~\cite{erdos1974some}, and Arratia~\cite{arratia1983site} has proven that $p^*(1)=1$.

The  rest of the paper is organized as follows. In section \ref{Sec:Model} we introduce the model, state our results, and give a summary of the proof construction.  
In section~\ref{Sec:Prelim} we provide a few preliminary results along with some results from previous works.
In section~\ref{Sec:Concentration} we develop the concentration bound for the spreading time of  the affected nodes.
In section~\ref{Sec:Shape_theorem} we use this concentration bound to obtain a shape theorem for the spread of affected nodes. 
In section~\ref{Sec:Proof} we prove a size theorem for the limiting configuration.

\section{Model and Main Results}\label{Sec:Model}

\subsection{Notational Conventions}
For any $d \geq 1$ and vector $v = (v(1),\ldots,v(d))$ we shall use the $l_\infty$ norm denoted by
\begin{align*}
\|v\|_\infty = \max_{1\le i\le d} |v(i)|.
\end{align*}
Also, we let
$\lfloor a \rfloor$ ($\lceil a \rceil$) be  the largest (smallest) integer $\le a \  (\ge a), a \wedge b = \min(a,b), a \vee b = \max(a,b)$. 

\subsection{The Model} 
\

\begin{definition}[Initial Configuration]
Consider the Cayley graph  
\begin{align*}
G_w = C(\mathbb{T}^2,\{-w,-w+1, \ldots ,w\}^2\setminus \{(0,0)\}),
\end{align*}
{ where $\mathbb{T} = \mathbb{Z}/h\mathbb{Z}$, $\log h=\omega(w^2)$, and  $w$ is assumed to be an integer. }We define the initial configuration by placing a binary particle at each node of  $G_w$  and choosing the state of each particle   independently at random to be $\theta$ or $\bar{\theta}$ according to a Bernoulli distribution of parameter $p  = 1/2$. By the state of a node we mean the state of the particle located at that node and by particle $u$ we mean the particle located at node $u$. 
\end{definition}

{ 
\begin{definition}[Neighborhood of a Node]    For any $u \in \mathbb{T}^2$, 
we  define a neighborhood of radius $\rho$, or $\rho$-neighborhood  of node $u$   as the set of all nodes in an $l_\infty$-ball of radius $\rho$ in the flat torus $\mathbb{T}^2$ centered at  $u$ as
\begin{align*}
\mathcal{N}_\rho(u) := \left\{v \in \mathbb{T}^2 :  \  \|v-u\|_\infty \le \rho  \right\}.
\end{align*}
When $\rho=w$ we drop the subscript and refer to \textit{the neighborhood of}  $u$ as  $\mathcal{N}(u)$. This also corresponds to a ball of unit radius in the graph metric centered at $u$.
The \textit{size} of a neighborhood is  defined as its cardinality.  We indicate the size of $\mathcal{N}(u)$  with $N=(2w+1)^2$.

\end{definition}

}

\begin{remark}
{ 
The choice $\log h = \omega(w^2)$ ensures that  while the size of $\mathcal{N}(u)$ is not uniformly bounded  as $w \rightarrow \infty$, it is exponentially small compared to the number of particles in the entire torus. This ensures that interactions among particles occur over  an unbounded range and  are also sufficiently ``local'' to enable us to study the limiting properties of the model.  On the other hand,  the choice of the $l_\infty$ norm is to comply with the original definition of a neighborhood in the Schelling model, but  results can be easily extended to other norms. } 
\end{remark}

\begin{definition}[Intolerance Parameter]
 We define the \emph{intolerance parameter}  as the rational number $\tau = \lceil \tilde{\tau} {\size} \rceil/{\size}$, where $\tilde{\tau} \in [0,1]$   and 
${\size} = (2w+1)^2$  is the size of the neighborhood of a node.    
\end{definition}

\begin{definition}[Stable, Unstable, and  p-Stable Particles] Let $s_\ti(u)$ denote the state of a particle at node $u$ at time $t$ and let 
\begin{align*}
r_\ti(u) = \frac{1}{\size} \sum_{v\in \mathcal{N}(u)} 1_{\{ s_\ti(u) = s_\ti(v) \}},
\end{align*}
where $1_{\{.\}}$ is the indicator function. A particle at node $u$ at time $t$ is  \textit{stable} if and only if $r_\ti(u) \ge \tau $.  A particle that is not stable is called \textit{unstable}. An unstable particle that can become stable by flipping its state is called  potentially stable, or \textit{p-stable}.
\end{definition}

\begin{definition}[Flipping Times] \label{Def:FlippingTime}
Every time a particle is labeled  p-stable, 
a new independent random variable with distribution $F$, denoted as its \textit{flipping time}, is assigned to it. We assume that $F$  satisfies the following properties: 
\begin{align} \label{eq:1.1}
F(x) = 0 \mbox{ for } x\le 0,\\
F \text{ is not concentrated on one point,} \label{eq:1.2} \\
\exists \gamma >0, \mbox{ such that } \int e^{\gamma x} F(dx) < \infty.  \label{eq:1.14}
\end{align}

The particle then waits for the amount of its last flipping time, and then flips its state, if   it has remained p-stable since the assignment of this flipping time.
\end{definition}

\begin{remark}  Two observations are  in order.  First, for $\tau< 1/2$, a particle that is unstable is also p-stable, however, this is not the case  for $\tau>1/2$. Second, when $F$ is exponential, the process dynamics are equivalent to a discrete-time model where at each discrete time step one p-stable particle is chosen uniformly at random and its state is flipped. 
\end{remark}

\begin{definition}[Affected Nodes] A node in $G_w$ is called $\theta$-\textit{affected} if a $\theta$-particle located there would be p-stable. By an affected node we mean a $\theta$-affected node.
\end{definition}

{ 

\begin{definition}[Final Configuration] We define  a \textit{final configuration} of the system as a
 configuration of   particles  where there are no p-stable particles.
 \end{definition}
 }
 
{ 
 \begin{remark}
 By defining a Lyapunov function to be  the sum over all nodes $u$ of the number of  particles in the same state as the particle at node $u$ that are present in its neighborhood, it is easy to argue that with probability one the process indeed reaches a final configuration, see \cite{barmpalias2018digital}.
\end{remark}

\begin{definition}[Monochromatic Ball]
At any point in time, the \textit{monochromatic balls} of a particle at node $u$  are the $l_\infty$-balls with largest radii that  contain only particles in a single state and  that  also  contain $u$. We choose one of these monochromatic balls arbitrarily and call it \textit{the monochromatic ball} of   node $u$.
\end{definition}
}

\begin{definition}[Size of Monochromatic Ball] 
The size of the monochromatic ball of  node $u$   at time $\ti$ is
\begin{align*}
M_\ti(u) := \sup_{\rho \in \mathbb{N}, v \in {\mathbb{T}^2}} \left| \left\{ \mathcal{N}_\rho(v) :  u\in  \mathcal{N}_\rho(v) \mbox{ and } \forall i \in  \mathcal{N}_\rho(v),\ \   s_\ti(i) = s_\ti(u)  \right\} \right|.
\end{align*}
The size of the monochromatic ball of a particle at node $u$ in the final configuration is denoted by $M(u)$. {  We also use $M_\ti$ and $M$ to denote $M_\ti(0)$ and $M(0)$ respectively.}
\end{definition}

Throughout the paper we say that an event occurs \textit{with high probability} (w.h.p.) if  its probability approaches one as $w$ tends to infinity. In all of our results,  the rate of   convergence of events that occur w.h.p.\ is always $1-o(w^{-2})$. 
We also say that an event occurs almost surely (a.s.) if the event occurs \textit{with probability equal to one} (w.p.1). 

\subsection{Main Results}  \label{Subsec:mainresults}
Our first result shows that 
conditional on having all the nodes in  a small neighborhood of the origin being affected,  this small neighborhood  w.h.p  ignites   a cascading process leading to more and more affected nodes, and this process  
creates a set of affected nodes whose shape during an exponentially large interval of time resembles a ball in a given metric. 
To state these results rigorously,  we   define the  set of affected nodes within  radius $\rho$ of the origin at time $\ti$ as follows:
\begin{definition}  
For any $\rho, \ti>0$, we define
\begin{align*}
A_{F,\rho}(0,\ti) := \left\{v \in \mathcal{N}_\rho(0) : v \mbox{ is affected at time } \ti \right\}.
\end{align*}
\end{definition} 

 We then consider the set $A_{F,
\rho}(0,\ti)$ for any 
\begin{align}
\ti & =\ti(w)  \in [2^{c_1\size} ,  2^{c_2\size}], \\
 \rho &=  \rho(w) = 2^{c'_2\size}, 
 \end{align}
 where $\size =(2w+1)^2$,  $0<c_1<c_2<c'_2<0.5(1-H(\tau'))$,    
 \begin{align}
 \tau' &= (\tau\size-2)/(\size-1),
 \end{align}
 and where
 \begin{align}
 H(\tau') &= -\tau' \log_2(\tau') - (1-\tau')\log_2(1-\tau')
 \end{align}
is the \textit{binary entropy function}. 
We also let  $\tau_* \approx 0.433$ be the solution of 
\begin{align} \label{Eq:tau_*}
\frac{3}{4} \left(1-H\left(\frac{4}{3} \tau_* \right)\right)- \left(1-H\left( \tau_* \right) \right) = 0. 
\end{align}


Given these choices,  we condition on having all the nodes in
$\mathcal{N}_{w/2}(0)$ being affected at time zero, and consider  any  $\tau \in (\tau_*,1/2)\cup (1/2,1-\tau_*)$.  The following theorem shows that in this case
w.h.p. there exists a norm $l_*(w)$ on  $\mathbb{R}^2$   such that, denoting by $B_{l_*}(0,\ti)$   the  ball of radius $\ti$ in     norm $l_*$ and centered at the origin, we have that   at time $\ti$  all  the   nodes in  $A_{F,\rho}(0,\ti)$ are contained in $B_{l_*}(0,\ti+o(\ti))$,  and  the   nodes in
$B_{l_*}(0,\ti-o(\ti)) $
 are all affected, i.e., they belong to $A_{F,\rho}(0,\ti)$.  The corresponding geometric picture is illustrated in Figure~\ref{fig:shape_theorem}. This shows the existence of  two concentric exponentially large balls centered at the origin, such that for $\ti(w)$ in the given interval 
 all the nodes in the affected set are contained within the outer ball, and all the nodes inside the inner ball   are  affected. 
 
 \begin{figure}[!t]
\centering
\includegraphics[width=3.2in]{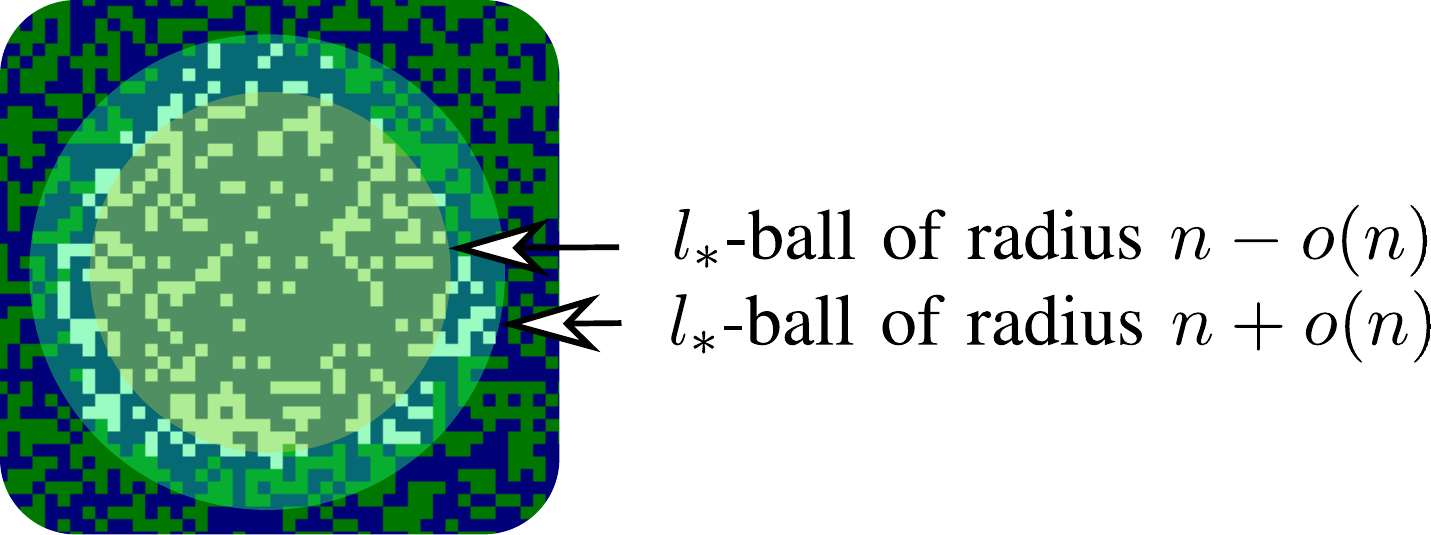}
\caption{W.h.p. there exists a norm $l_*$ on the plane, such that, at time $\ti$ all the affected nodes in $A_{F,\rho}(0,\ti)$  are contained in $B_{l_*}(0,\ti+o(\ti))$,  and all  the   nodes in  ${B_{l_*}(0,\ti-o(\ti))}$ are affected.  }
\label{fig:shape_theorem}
\end{figure}

\begin{theorem*}[Shape Theorem ---Transient]\label{Thrm:Shape_theorem}  For all $\tau \in (\tau_*,1/2)\cup (1/2,1-\tau_*)$,  $t = \ti(w) \in [2^{c_1\size} ,  2^{c_2\size}]$, and $\rho(w)=2^{c'_2\size}$, where $c_1, c_2, c'_2$, and $N$ are as defined above, and conditional on having all the nodes in $\mathcal{N}_{w/2}(0)$ being affected at time $t=0$, w.h.p. there exists a norm $l_*(w)$ on $\mathbb{R}^2$, and a  constant $c>0$, such that at time $\ti$ we have
\begin{align}\label{eq:Tess0}
B_{l_*}(0,\ti-\size^c\ti^{1/2}\log^{3/2} \ti)\cap \mathbb{Z}^2 \subset A_{F,\rho}(0,\ti) \subset B_{l_*}(0,\ti+\size^c \ti^{1/2}\log^{3/2} \ti).
\end{align}

\end{theorem*}
{ 
\begin{remark} 
In Theorem~\ref{Thrm:Shape_theorem} without loss of generality  we have assumed that the linear time scale is  chosen such that the re-scaled  limit shape of the set $A_{F,\rho}(0,\ti)$  is a unit $l_*$-ball.  
The lower and upper bounds   $2^{c_1  N} \leq t \leq 2^{c_2N}$ in this theorem reflect the following trade-off: on the one side the time scale $\ti$ is chosen large enough so that after time $\ti$ the set of affected nodes roughly resembles a ball. On the other side, the  time scale, and therefore the size of the ball, is chosen small enough to ensure that the region of space where the spreading process occurs is  sufficiently clear of any   $\theta$ or $\bar{\theta}$-affected nodes so that the process can proceed without interference.





Since as $w \rightarrow \infty$  we also have that $\ti \rightarrow \infty$, it can be of interest to investigate the limiting behavior of the norm $l^*(w)$.  Although numerical simulations suggest that $l_*$   may converge to the Euclidean norm,  this question remains open.
Finally, we also mention that similar shape theorems  have been proven in the literature for percolation models and other contact processes~\cite{kesten1986aspects,kesten1993speed,alves2002shape,drewitz2014chemical,tessera2014speed}, but none of them applies to our model. 
\end{remark}
}

To state our second result, 
we let   $\tau^* \approx 0.488$ be the solution of
\begin{align}
5\left(1+e(\tau)\right)^2 - 6 = 0,
\end{align}
where 
\begin{align}
e(\tau) = \frac{3(\tau - 0.5)+\sqrt{9(\tau-0.5)^2-7(\tau-0.5)(3\tau+0.5)}}{2(3\tau+0.5)}.
\label{eq:ftau}
\end{align}



{ 

\begin{theorem*}[Size Theorem ---Final Configuration]\label{Thrm:main_thrm}
For all   $\tau \in (\tau^*,1-\tau^*) \setminus \{1/2\}$, let 
\begin{align*}
a(\tau) = \left(1-H(\tau)\right) \left(2-(1+e(\tau))^2\right),
\end{align*}
and 
\begin{align*} 
b(\tau)= \left(1+e(\tau)\right)^2\left(1-H(\tau)\right).
\end{align*}
For all $\epsilon > 0$, w.h.p. we have 
\begin{align}
 2^{\left(a(\tau)-\epsilon\right){\size}} \le M \le 2^{\left(b(\tau)+\epsilon\right){\size}}.
\end{align}
\end{theorem*} 
}

The numerical values for $a(\tau)$ and $b(\tau)$  are plotted in Figure~\ref{fig:aaprime}. 

\begin{remark} For $\tau \in (\tau^*,1-\tau^*)\setminus \{1/2\}$,  as the intolerance gets farther from one half in both directions, larger monochromatic balls are formed w.h.p.
 An intuitive  explanation for this behavior is that as $\tau$ decreases farther from $1/2$, agents become more tolerant and  configurations that can start a cascading process  become less likely, and hence  located farther from each other.  The cascading process can then evolve without interference from other cascading processes for a longer amount of time, and lead to the formation of larger monochromatic balls. On the other hand, as $\tau$ increases farther from $1/2$ agents become less tolerant. Since unstable agents will flip their state only if this makes them stable, configurations that can start a cascading process  become less likely and located farther from each other in this case as well. By the same argument as above, once the cascading process is ignited, it leads to larger monochromatic balls.
\end{remark}

\subsection{The Infinite Lattice Case}
We can consider similar dynamics occurring on the infinite lattice $\mathbb{Z}^2$ instead of the finite torurs $\mathbb{T}^2$. 
 Using Theorem B3 of \cite[p.~3]{liggett2013stochastic} it is easy to verify that  the process on  $\mathbb{Z}^2$   exists, and is unique, and is a Feller Markov process on $\{{\theta},\bar{\theta}\}^{\mathbb{Z}^2}$.
The following corollary follows from the proof of Theorem~\ref{Thrm:main_thrm}.

\begin{corollary*}[Size Theorem ---Infinite Lattice]\label{Corr:size_thrm}
 For all $\epsilon>0$, $\tau \in (\tau^*,1-\tau^*) \setminus \{1/2\}$, let $\ti^* = 2^{{(a(\tau)+\epsilon)}\size}$. For all $\ti \ge \ti^*$, w.h.p. we have 
\begin{align} 
 2^{\left(a(\tau)-\epsilon\right){\size}} \le M_\ti \le  2^{\left(b(\tau)+\epsilon\right){\size}}.
\end{align}
\end{corollary*}


 \begin{figure}
\centering
\includegraphics[width=3.5in]{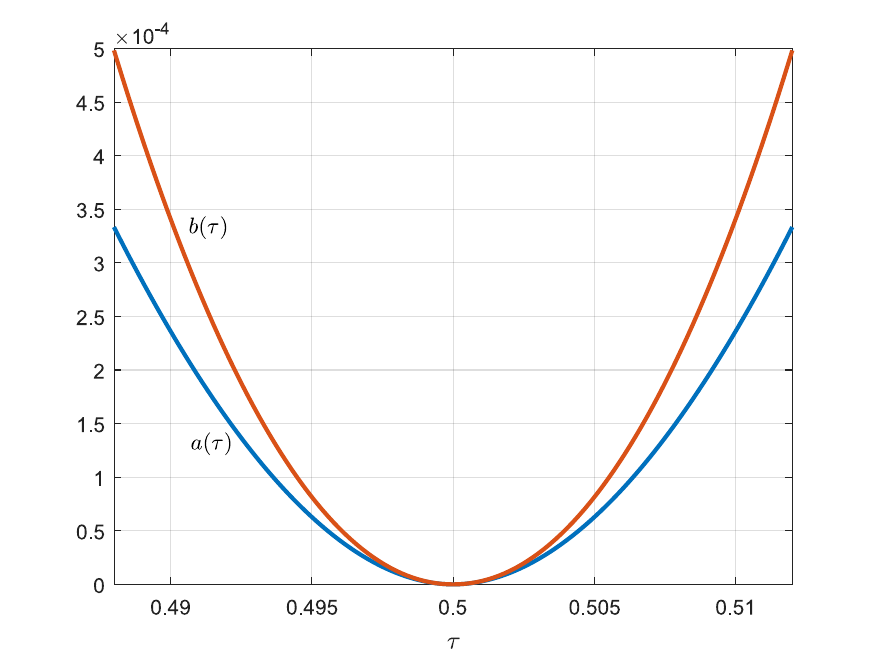}
\caption{Exponent multipliers $a(\tau)$ and $b(\tau)$ for the lower bound and upper bounds on the size of the largest monochromatic balls.} 
\label{fig:aaprime}
\end{figure}


\subsection{Proof Outline}
\

\textit{Shape Theorem.}
 To prove the shape theorem, we   consider a variation of our process on $G_w$.  We describe the spreading of the affected nodes around the origin for this new process and then use a coupling argument to show that the same spreading also occurs in the original process. To describe the spreading in the new process, 
we adapt a strategy developed by Tessera \cite{tessera2014speed} for first passage percolation (FPP). However, while Tessera's result relies on Talagrand's concentration inequality \cite[Proposition~8.3]{talagrand1995concentration} (also restated as Theorem~1 in \cite{tessera2014speed})  for the spreading time in FPP, our shape theorem is based on a concentration bound that we develop independently for the spread of affected nodes in our new process, by extending some results of Kesten's  \cite{kesten1986aspects,kesten1993speed}. This bound is a key step in our proof.

We  show that the set of affected nodes at time $\ti$ in the new process is  close  to the set of nodes whose expected time of becoming affected is at most $\ti$, and that there exists a  norm $l_*$ such that the latter set is also close to an $l_*$-ball of radius $\ti$. The former statement is proved using the concentration bound (our Theorem~\ref{Thrm:Theorem_Concentration}), and the latter statement is proved  showing that the expected time of the spread from one point to another in $G_w$ is asymptotically close to the length of a corresponding ``optimal geometric path" (again using Theorem~\ref{Thrm:Theorem_Concentration}).

To derive the bound on the spreading time of the affected nodes, we represent the difference between the random spreading time and the mean spreading time between any two nodes as a sum of martingale differences and, after estimating the sum of squares of these differences, apply a martingale inequality developed by Kesten (re-stated as Theorem~\ref{Thrm:Kesten_Azuma} of this paper). Along the way, to bound the  martingale differences, we use a modified result from \cite{kesten1986aspects} to compare our process with two FPP processes on $\mathbb{Z}^2$.

\

\textit{Size Theorem.}
The main idea of the proof in this case is to show that w.h.p. while the spread of the $\theta$-affected nodes reaches the origin, the $\bar{\theta}$-affected nodes  are still at distances at least exponential in $\size$ from the origin.  Once the origin is reached, the p-stable particles around it  will w.h.p. lead to the formation of an exponentially large ``firewall" (i.e., an annulus of particles in the same state)  that is indestructible by other spreading processes. The interior of this firewall will then become monochromatic, so that in the final configuration there will be  w.h.p. an exponentially large monochromatic ball around the origin.

To elaborate on this main idea, we define an \textit{expandable region}, that  is composed of a local configuration of particles and a possible set of flips inside it, that can lead to at least one new affected node outside of it.  We consider the expandable region closest to the origin in the $l_*$ norm, and denote its type by $\theta$. We denote its $l_*$-distance to the origin by $X$, and  consider an $l_*$-ball of radius $X$ at the origin. We then argue that, since there are no expandable regions in this ball, any spreading of affected nodes of any type inside this ball dies out quickly, while  the expandable region starts a spreading of $\theta$-affected nodes towards the origin w.h.p. 
We then find an upper bound $X \leq \rho$, where $\rho=\rho(\tau,\size)$, that holds w.h.p.,  and choose $\rho'=\rho'(\tau,\size)$ such that  w.h.p there is no $\bar{\theta}$-expandable region inside the annulus $B_{l_*}(0,X+\rho') \setminus B_{l_*}(0,X)$.
We  consider the extremal case $X=\rho$ and study the ``race" between the possible spreads of the $\bar{\theta}$-affected nodes from outside $B_{l_*}(0,\rho+\rho')$ and the spread of the expandable region at distance $\rho$ towards the origin,  see Figure~\ref{fig:gradual_growth}.  

 \begin{figure}[!t]  
\centering
\includegraphics[width=3.7in]{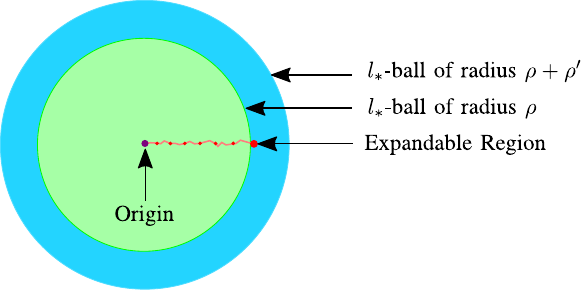}
\caption{The gradual spread of affected nodes from the expandable region towards the origin in time increments of size $\rho'/4$.} 
\label{fig:gradual_growth}
\end{figure} 

We consider the gradual spreading of the expandable region in time intervals of $\rho'/4$ towards the origin. Using the shape theorem, we argue that w.h.p. the origin will eventually be contained in a neighborhood such that all of its nodes are $\theta$-affected. We then show that  the origin is quickly surrounded by  an exponentially large \textit{firewall} while any spreading of affected nodes started from outside $B_{l_*}(0,\rho+\rho')$  is still at large distances from it. This firewall is an indestructible monochromatic annulus which isolates the origin from the outside flips, see Figure~\ref{fig:fw0}. It  will thus protect the cascading process which w.h.p. leads to the formation of a monochromatic ball of size exponential in $\size$ containing the origin. This shows that the lower bound occurs w.h.p. 
To see that the upper bound also occurs w.h.p. we note that in a large enough exponential size neighborhood around the origin, w.h.p. the origin will be surrounded by exponentially large monochromatic balls of particles in both states protected by firewalls. 
\

In our proofs throughout the paper, we focus on the case where $\tau < 1/2$. The results for $\tau > 1/2$ follow by a simple argument provided in Section~\ref{Subsec:Extension}.

\begin{figure}[!t] 
\centering
\includegraphics[width=1.6in]{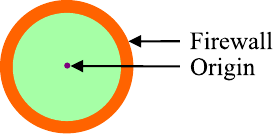}
\caption{A firewall formed around the origin.} 
\label{fig:fw0}
\end{figure}

\section{Preliminary and Previous Results}\label{Sec:Prelim}
We begin with the following elementary lemma giving lower and upper bounds for the probability of a node being affected.

\begin{lemma*}\label{Lemma:unstableprob}
Let $p_u$ be the probability of being $\theta$-affected for an arbitrary node in the initial configuration. There exist positive constants $c_l$ and $c_u$ which depend only on $\tau$ such that
\begin{align*}
c_{l}\frac{2^{-[1-H(\tau')]{\size}}}{\sqrt{{\size}}}  \le p_u \le c_{u}\frac{2^{-[1-H(\tau')]{\size}}}{\sqrt{{\size}}},
\end{align*}
 where $\tau' = \frac{\tau {\size} - 2}{{\size}-1}$, and $H$ is the binary entropy function.
\end{lemma*}

\begin{proof}
We have 
\begin{align} \label{eq:pu}
p_u = \frac{1}{2^{\size-1}} \sum_{k = 0 }^{\tau {\size} - 2}{{{\size}-1}\choose{k}},
\end{align}
where the two unit reduction is to account for the strict inequality and the particle at the node itself at the center of the neighborhood. Let $\tau' = \frac{\tau {\size} - 2}{{\size}-1}$. After some   algebra, we have
\begin{align*}
{{\size}-1 \choose \tau' ({\size}-1)} \le \sum_{k = 0 }^{\tau' ({\size}-1)}{{{\size}-1}\choose{k}} \le \frac{1-\tau'}{1-2\tau'}{{\size}-1 \choose \tau' ({\size}-1)},
\end{align*} 
and using Stirling's formula, there exist constants $c,c' \in \mathbb{R}^+$ such that
\begin{align*}
{c\frac{2^{H(\tau')({\size}-1)}}{\sqrt{({\size}-1)\tau'(1-\tau')}}} \le {{\size}-1 \choose \tau' ({\size}-1)} \le c'\frac{2^{H(\tau')({\size}-1)}}{\sqrt{({\size}-1)\tau'(1-\tau')}}.
\end{align*}
The result follows by combining the above inequalities.  
\end{proof} 

The following lemma is a consequence of Lemma~\ref{Lemma:unstableprob}.
\begin{lemma*} \label{Lemma:R_unhappy}
Let ${\radius}  =  2^{c\size}$ where $0<c<[1-H(\tau')]/2$ and $\tau'$ is as defined in Lemma~\ref{Lemma:unstableprob}. In the initial configuration, the following event occurs w.h.p. 
\begin{align*}
A = \left\{\nexists \  \theta\mbox{-affected nodes in } \mathcal{N}_{\rho}\right\}. 
\end{align*}
\end{lemma*}

\begin{definition}[$m$-Blocks and Monochromatic Blocks]
We define an $m$-\textit{block} to be an $m/2$-neighborhood. A \textit{monochromatic block} is a block whose particles are all in the same state. When $m$ is not specified, by a \textit{block} we mean a $w$-block.
\end{definition}

\begin{definition}[Region of Expansion]
 We call  a \emph{region of expansion} of type $\theta$ any neighborhood whose configuration is such that     if we change all the $\theta$-particles in a $w$-block to $\bar{\theta}$, then all the $\theta$-particles on its outer boundary (i.e., the set of particles in the set composed by the $(w+2)$-block co-centered with the $w$-block and excluding the $w$-block itself) are p-stable.
\end{definition}
  The next lemma is a restatement of Lemma 8 in \cite{omidvar2017self2}.

\begin{lemma*}[\cite{omidvar2017self2}]  \label{Lemma:monoch_spread_1}
Let $\tau \in (\tau_*,1/2)$ and let $\mathcal{N}_r$ be any neighborhood with radius $r \in (w+2,2^{c{\size}})$ where $c\in (0,0.5{[1-H(\tau')]})$ in the initial configuration.  Then, $\mathcal{N}_r$ is a region of expansion w.h.p.
\end{lemma*} 

{
\begin{remark}
Equation (\ref{Eq:tau_*}), which  provides the value of $\tau_*$, is derived in the proof of the above lemma. 
\end{remark}
}

Let $\mathcal{M}$ denote an arbitrary $m$-block with $m\ge w$ and
{ 
\begin{align*}
\mathcal{I}_{\mathcal{M}} := \bigcup_{v\in \mathbb{T}^2} \left\{ \mathcal{M} \cap \mathcal{N}_{w/2}(v) \right\}.
\end{align*} 
}
Also, let $W_{I}$ be the random variable representing the number of particles in state $\bar{\theta}$ in $I \in \mathcal{I}_{\mathcal{M}}$, and ${\size}_I$ be the total number of particles in $I \in \mathcal{I}_{\mathcal{M}}$.

\begin{definition}[Good Block]
 For any $\epsilon \in (0,1/2)$, $\mathcal{M}$ is called a \textit{good $m$-block} of type $\theta$, if and only if  for all $I\in \mathcal{I}_{\mathcal{M}}$ we have $W_I-{\size}_I/2  < {\size}^{1/2+\epsilon}$. Otherwise $\mathcal{M}$ is called a \textit{bad $m$-block} (see Figure~\ref{fig:goodbad}). 
\end{definition} 
 
 By the following lemma (which is a restatement of Lemma 11 in~\cite{omidvar2017self2}), an $\msize$-block is a good block w.h.p. Since   for sufficiently large $\size$ the number of  particles in different states inside a good block is ``balanced'', a node whose entire neighborhood is contained in a good block cannot be a $\theta$-affected node.

 \begin{figure}[!t]
\centering
\includegraphics[width=2in]{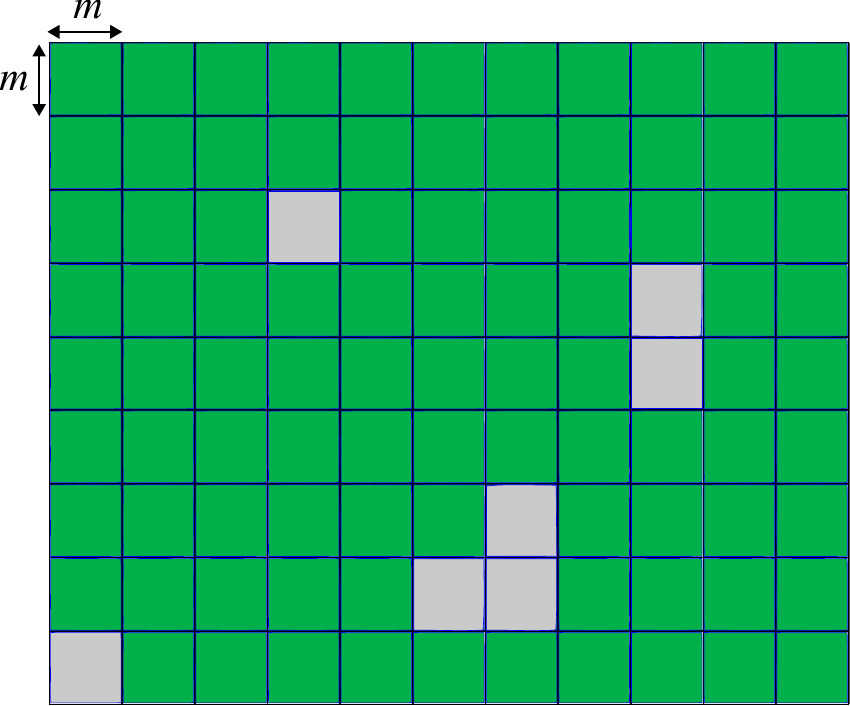}
\caption{Part of the lattice partitioned into $m$-blocks.  Green and gray indicate good and bad blocks respectively. }
\label{fig:goodbad}
\end{figure}

\begin{lemma*} \label{Lemma:goodblock}
Let $\epsilon \in (0,1/2)$. There exists a constant $c>0$, such that, in the initial configuration, for all $I\in \mathcal{I}_{\mathcal{M}}$ we have
\begin{align*} 
W_I-{\size}_I/2 < {\size}^{1/2+\epsilon}
\end{align*} 
with probability at least
\begin{align*}
1 - e^{-c{\size}^{2\epsilon}}.
\end{align*} 
\end{lemma*}

We now want to review a result from percolation theory. Without loss of generality, we assume $G_w$ is defined on $\mathbb{Z}^2$ rather than on the finite torus $\mathbb{T}^2$.  Let us re-normalize (i.e., partition and rescale) $G_w$ into $m$-blocks starting from the $m$-block centered at the origin (see Figure~\ref{fig:goodbad}) and define the graph $G'_w := (\mathcal{V}',\mathcal{E}')$ where the set of vertices $\mathcal{V}'$ is the set of all $m$-blocks and the set of edges $\mathcal{E}'$ is the set of all pairs of $m$-blocks that are horizontally, vertically, or diagonally adjacent. By a \textit{bad node} in $G'_w$ we mean a node corresponding to a bad $m$-block. We define a path in $G'_w$ as an ordered set of bad nodes in this graph such that each pair of consecutive nodes are adjacent in $G'_w$ and no node appears more than once in the set.   Given any bad node $x$, we denote by $S(x)$ the set composed of all nodes for which there is a path to node $x$. 
The boundary of $S(x)$, denoted by $\partial S(x)$, is the set of nodes in $S(x)$ which are adjacent to some node not in $S(x)$.
The distance $\Delta(y,z)$ between any two nodes in $S(x)$  is the cardinality  of the shortest path  that connects them. 
Let $\partial S_k(x)$ be the set of nodes on the boundary of $S(x)$ that are at distance   $k$ from $x$ and $A_k(x)$ be the event that there exists a path of bad nodes joining $x$ to some node in $\partial S_k(x)$. 

The \textit{radius} of   $S(x)$ is the supremum of the distance from $x$ to the boundary, namely 
\begin{align*}
\sup \{\Delta(x,y): y \in \partial S(x)\},
\end{align*}
so that the event $A_k(x)$ indicates that the radius of $S(x)$ at $x$ is at least $k$.
Let $p$ denote the probability that an arbitrary $m$-block in the initial configuration in the above setting is a bad $m$-block.  It is noted that an $m$-block is a bad $m$-block independently of the others. Let $0<p_c<1$ denote the critical probability in the above percolation setting. The following result is Theorem 5.4 in \cite{grimmett1999percolation}.

\begin{theorem*}\label{Thrm:grimmett_bad_cluster} 
 \textsl{(Exponential tail decay of the radius of a cluster.)}
If $p<p_c$, then there exists $\psi(p) >0 $ such that   for all  $k$
\begin{align*} 
P_p\left(A_k(x)\right) < e^{-k\psi(p)}. \ \ \  
\end{align*}
\end{theorem*}


\begin{definition}[Firewall]
A \textit{firewall} of radius $r$ and center $u$ is a monochromatic annulus
\begin{align*}
A_r(u) = \left\{y: r-\sqrt{2}w \leq \|u-y\|_2 \leq r\right\},
\end{align*}
    where $\|.\|_2$ denotes the Euclidean distance and $r\ge 3w$.
\end{definition}

Consider a disc of radius $r$,   centered at a particle such that all the particles inside the disc are in the same state. 
Lemma 6 in \cite{immorlica2015exponential} shows that if $r>w^3$,  $\tau \in (\tau^*,1/2)$,  and $w$ is sufficiently large, then all the particles inside the disc will remain a.s.\ stable regardless of the configuration of the particles outside the disc. Here we state a similar lemma but for a firewall, without proof.


\begin{lemma*}[\cite{immorlica2015exponential}]  \label{Lemma:firewall}
Let $A_r(u)$ be the set of particles contained in an annulus of outer radius $r \ge w^3$ and of width $\sqrt{2} w$ centered at $u$. For all $\tau \in (\tau^*,1/2)$ and for a sufficiently large constant $w$, if $A_r(u)$ is monochromatic at time $\ti$,  then it will a.s.\ remain monochromatic at all times  $\ti'>\ti$. 
\end{lemma*}

 By Lemma~\ref{Lemma:firewall}, once formed a firewall of sufficiently large radius  remains  static,  and since its width is  $\sqrt{2} w$ the particles inside the inner circle are not going to be affected by the configurations outside the firewall.

\

We now review some of the definitions from \cite{omidvar2017self2}. In that paper, the goal is to identify a configuration that can trigger a cascading process leading to the formation of monochromatic balls w.h.p. 

\begin{definition}[Radical Region]
For any $\epsilon, \epsilon' \in {(0,1/2)}$ let
\begin{align*}
\hat{\tau} = \tau (1- 1/ (\tau {\size}^{1/2-\epsilon})),
\end{align*}
and define  a
\textit{radical region}  to be a neighborhood 
of radius ${(1+\epsilon')w}$ containing  less than $\hat{\tau}(1+\epsilon')^2{\size} $  particles in state $\theta$. 
\end{definition}


\begin{definition}[p-Stable Region] For any $\epsilon, \epsilon' \in {(0,1/2)}$, we define
a \textit{p-stable region} to be a neighborhood of radius $\epsilon' w$, containing at least $\lfloor \tau \epsilon'^2 {\size}-{\size}^{1/2+\epsilon} \rfloor$ p-stable particles in state $\theta$. 
\end{definition}

The following lemma easily follows from Lemma 4 in \cite{omidvar2017self2} and is given without proof.
\begin{lemma*} \label{Lemma:unstable_region}
A radical region $\mathcal{N}_{(1+\epsilon')w}$ in the initial configuration contains a p-stable region $\mathcal{N}_{\epsilon'w}$ at its center w.h.p.
\end{lemma*}



Now consider a geometric configuration 
where a radical region,  and  neighborhoods $\mathcal{N}_{\epsilon'w}$ , $\mathcal{N}_{w/2}$ and $\mathcal{N}_{{\radius}}$  with ${\radius}>3w$, are all co-centered. 
Let 
\begin{align} 
\mathcal{T}({\radius}) = \inf\{\ti: \exists v\in \mathcal{N}_{\radius}, \;  v \mbox{ is a $\bar{\theta}$-affected node at time } \ti \}.
\label{Tinfdef}
\end{align}


\begin{figure}[!t]
\centering
\includegraphics[width=4.5in]{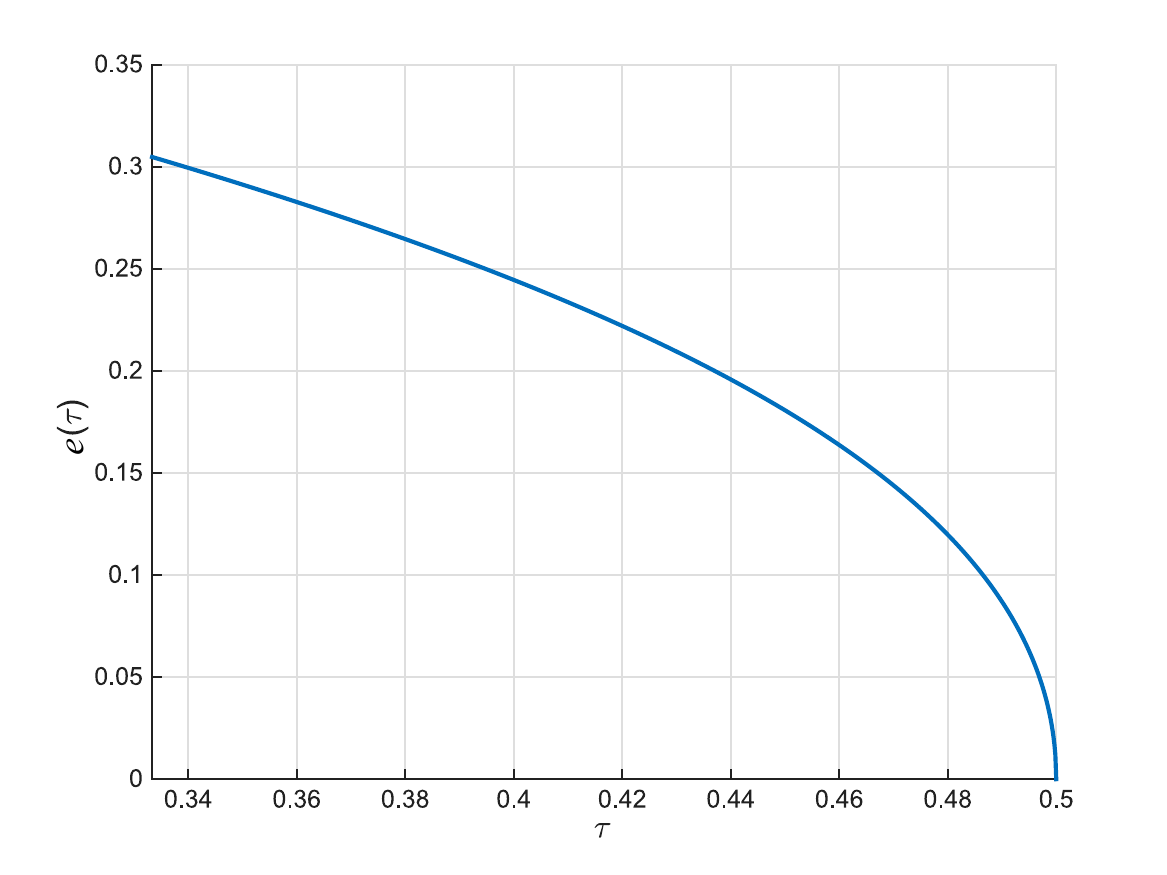}
\caption{The function $e(\tau)$ gives the infimum of $\epsilon'$ to potentially trigger a cascading process.}
\label{fig:epsilon} 
\end{figure}

\begin{definition}[Expandable Radical Region] \label{Def:exradreg}
A radical region is called an \emph{expandable} radical region (of type $\theta$) if there is a possible sequence of at most $(w+1)^2$ flips inside it that can make  the neighborhood $\mathcal{N}_{w/2}$ at its center monochromatic with particles in state $\bar{\theta}$. 
\end{definition}
The next lemma, which is a restatement of Lemma~5 in \cite{omidvar2017self2},  shows that a radical region in this configuration is an expandable radical region of type $\theta$ w.h.p., provided that  $\epsilon'$ is large enough and there is no $\bar{\theta}$-affected node in $\mathcal{N}_{\radius}$.  The main idea is that the $\bar{\theta}$ particles in the p-stable region at the center of the radical region can  trigger a process that leads to  monochromatic balls of radius $w$.

\begin{lemma*} \label{Lemma:Trigger}
For all $\epsilon' > e(\tau)$
there exists w.h.p. a sequence of at most $(w+1)^2$ possible flips  in a radical region $\mathcal{N}_{(1+\epsilon')w}$ such that if they  happen  before $\mathcal{T}({\radius})$,  then all the particles inside   $\mathcal{N}_{w/2}$ will have the same state.
\end{lemma*}

The function $e(\tau)$ is plotted in Figure~\ref{fig:epsilon}.
The following lemma, which is Lemma~20 in \cite{omidvar2017self2},  gives a lower bound and an upper bound for the probability that an arbitrary neighborhood of the size of a radical region is indeed a radical region in the initial configuration.

\begin{lemma*}[\cite{omidvar2017self2}] \label{Lemma:radical_region_prob}
Let $\epsilon'$ and $\epsilon$ be positive constants. In the initial configuration, an arbitrary neighborhood with radius $(1+\epsilon')w$ is a radical region with probability $p_{\epsilon'}$, where
\begin{align*}
{2^{-[1-H(\tau'')](1+\epsilon')^2{\size}-o({\size})} \le p_{\epsilon'} \le  2^{-[1-H(\tau'')](1+\epsilon')^2{\size}+o({\size})}},
\end{align*}
and $\tau'' = (\lfloor \hat{\tau}(1+\epsilon')^2{\size} \rfloor - 1) / (1+\epsilon')^2{\size} $, $\hat{\tau} = (1-{1}/{(\tau {\size}^{1/2-\epsilon})})\tau$, and $H$ is the binary entropy function.
\end{lemma*}

\section{The Concentration Bound} \label{Sec:Concentration}
\
The concentration bound developed in this section is a main step required for proving our main results and it resembles the one  developed in the context of FPP \cite{kesten1993speed}.

\subsection{First Passage Percolation} 
FPP was originally developed by Hammersley and Welsh~\cite{hammersley1965first} in 1965 and then further developed especially by Kesten in 1986 \cite{kesten1980critical, kesten1986aspects, kesten1993speed}. It is a mathematical model  to describe the flow of a fluid through a random medium.   A non-negative random variable $t_i$, defined as the passage time of an edge $e_i$, is placed at each nearest-neighbor edge of the grid defined on $\mathbb{Z}^d$. The collection $\{t_i\}$ is assumed to be independent, identically distributed with a common distribution $F$. The random variable $t_i$ is interpreted as the time or the cost needed to traverse edge $e_i$. The basic problem is to asymptotically describe the set of nodes that have been visited by a given time, when a fluid starts traversing the edges from the node at the origin.   FPP on the nodes of the lattice can also be defined along the same lines, when passage times are assigned to nodes instead of  edges. 
\

The beauty of Kesten's approach in developing a concentration bound for FPP, is that he uses a martingale representation for the difference between the first passage time of any node and its mean that is independent of the underlying geometrical structure of the process. We  exploit  such a property in our proof.

\subsection{Modified Model}
In this section, we work with a  modified  version of the model introduced in Section~\ref{Sec:Model}. We prove some results for this new model, and then use them to prove Theorem~\ref{Thrm:Shape_theorem}   by coupling the modified model with  our original model.  
The modified model is defined in a way that the spreading of the affected nodes    resembles a FPP model.
\begin{figure}[!t]
\centering
\includegraphics[width=2.5in]{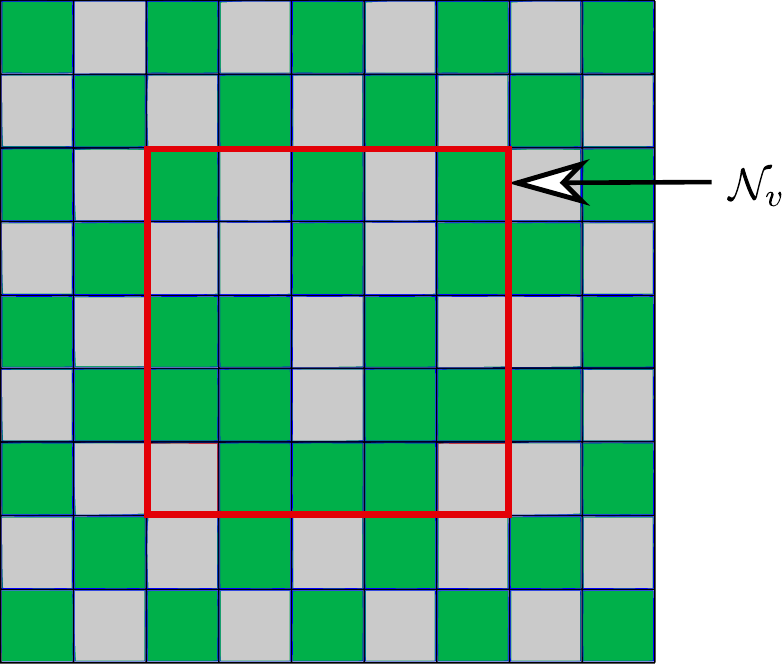}
\caption{The function $e(\tau)$ gives the infimum of $\epsilon'$ to potentially trigger a cascading process.}
\label{fig:Nv} 
\end{figure}
The model differs from the original one in the following points:

\begin{itemize}
\item We consider $G_w$ to be constructed on an infinite integer lattice $\mathbb{Z}^2$. Let $c_2,c'_2, \tau'$ be as defined in Section~\ref{Subsec:mainresults}. We assume that the states of particles that are located outside a neighborhood $\mathcal{N}_v$ with radius $v=2^{c'_2\size}$ centered at the origin, are frozen at the initial configuration for the entire duration of the process.  

\item  
We consider the product space for the initial states of the particles and we denote by $V$ the event that in the initial configuration there are no affected nodes in $\mathcal{N}_v$ and this neighborhood is a region of expansion. Using Lemma~\ref{Lemma:R_unhappy} and Lemma~\ref{Lemma:monoch_spread_1}, we have that $V$ occurs w.h.p. It follows that without loss of generality, to prove our asymptotic results  we can assume that in the modified model all   events  in the initial configuration  are conditional on $V$.
 
\item We assign a single flipping time to every particle on $G_w$ in the initial configuration and  
%
we label some nodes as  \textit{affected*}, as indicated below. 
In our modified model $\theta$-particles placed on  affected* nodes always flip, once their flipping time expires, regardless of them being p-stable or not. On the other hand, particles on nodes that are not affected*, upon becoming p-stable they  wait    for an amount of time equal to their flipping time and then make a flip. The first two bullets above   ensure that   if a particle becomes p-stable it   remains p-stable until it flips, and that once flipped any particle remains stable forever.

\item We define an \textit{affected* block} to be  a $w$-block such that all the nodes on it are affected* and we place such a block at a given location on $\mathcal{N}_v$. Namely, in the initial configuration we 
label the nodes located in this block as affected*. 

\end{itemize}


\subsection{Additional Definitions for the Modified Model}
 We   now define  notions of paths and passage times similar to FPP. 
\begin{definition}[Path and First Passage Time for the Modified Model]
Let $t_i$ be the flipping  time  of the particle located at $v_i$. Let $\mathbf{0^*}$ denote an affected* block placed at the origin, and $u$ an arbitrary node on $\mathcal{N}_v$. 
 Let a path $r$ be a set of particles $v_1,v_2,\ldots,v_k$  such that they become p-stable in a sequence, i.e., $v_j$ becomes p-stable after $v_{j-1}$, and let
\begin{align*} 
T(r) = \sum_{i=1}^{k} t_i. 
\end{align*}
Then, we define the \textit{first passage time} for the modified model from the origin to $u$ as
\begin{align*}
a_{0,u} := T(\mathbf{0^*},u) :=  \inf_{r\in \mathcal{P}} \{T(r) \},
\end{align*}
where $\mathcal{P}$ is the set of possible sequences of flips when we place an affected* block at the origin that will lead to an affected node at $u$. 
When we assume that in the initial configuration the entire $\mathcal{N}_v$ is a region of expansion and there are no affected nodes of any type on $\mathcal{N}_v$, $\mathcal{P}$ will be non-empty. From the definition of region of expansion, since the affected* block will lead to the formation of a monochromatic block, there will always be a possible subsequent flip until the desired node becomes affected.
Clearly, the notion of first passage time can be extended to any pair of nodes on a neighborhood $\mathcal{N}_v$ denoted by $a_{x,u}$, in which case, we assume that the affected* block is placed at $x$ instead of the origin. 
We note that with the given conditions on the initial configuration   and if we place an affected* block  at $x$, with a similar argument as above there always exists a set of flips leading to an affected node at any location on $\mathcal{N}_v$.

%
\end{definition}

\subsection{The Concentration Bound}

\begin{theorem*}\label{Thrm:Theorem_Concentration}
Let $c>0$ be a  constant such that $c\le c_2$ where $c_2$ is as defined in Section~\ref{Subsec:mainresults}, and let $u$ be any node on $G_w$ whose $l_\infty$-distance from the origin is at least $2^{c\size}$.
There exists a constant $c'>0$ (independent of $\size$) such that, for $\size$ sufficiently large and for all $\lambda\le \|u\|$,  we have
\begin{align} \label{eq:1.15} 
P\left\{\left|a_{0,u}-\mathbb{E}[a_{0,u}] \right| \ge \lambda    \right\} \le e^{-c'\lambda/\sqrt{\|u\|\log \|u\|}}. 
\end{align}
\end{theorem*}

\subsection{Proof Outline}

To prove the concentration bound, we represent the difference between $a_{0,u}$ and $\mathbb{E}[a_{0,u}]$ as a sum of martingale differences. The principal step to bound this sum   is to first obtain a bound for the sum of the squares of these differences and apply  a martingale inequality developed by Kesten~\cite{kesten1993speed}  (re-stated as Theorem~\ref{Thrm:Kesten_Azuma} of this paper). 
This   boils down to finding upper bounds for two probabilities. The first one is the probability of the existence of a self-avoiding path with at least a given number of steps and with at most a given passage time (see (\ref{eq:2.23}) and (\ref{eq:2.26})), the second one is the probability that the first passage time is larger than a certain quantity. To bound these terms, we introduce two intermediate FPPs that,  when coupled with our modified model, provide upper bounds for each  of these probabilities. 
 We decorate the variables pertaining to these two processes by prime and double prime respectively, throughout this section. 

\subsection{First intermediate FPP process}

First consider the following auxiliary process. Put a $\theta$ particle on each  node of $\mathbb{Z}^2$. we partition $\mathbb{Z}^2$ into $w$-neighborhoods starting from the $w$-neighborhood centered at the origin and consider the re-normalized lattice $Z'$  where each node of $Z'$ corresponds to a  $w$-neighborhood of the partition. We assume that each node of $Z'$ is connected to its horizontal, vertical, and diagonal neighbors.  Initially  all the particles on $\mathbb{Z}^2$ in this process, regardless of their neighborhood configuration,  are labeled as stable, except for the particles located in the $w$-neighborhood centered at the origin, which are labeled as p-stable.
Each particle is assigned a single flipping time $t_i$, that is the same (single) flipping time of the particle in our modified process. 
Once a particle $v_i$ flips, it is labeled as stable permanently, and all the $\theta$-particles  in  the $w$-neighborhoods that are horizontally, vertically, or diagonally adjacent to the $w$-neighborhood where the flip has occurred, i.e., all the $\theta$-particles on the neighboring $w$-neighborhoods in $Z'$, are labeled as p-stable. We then define a path as a sequence of $v_i$'s such that the particles located on them can flip in this order. By coupling, 
it should be clear that 
 if there is  a self-avoiding path with its first node located in a block centered at the origin with at least a given number of steps and at most a given passage time in our modified process, then there is such a path also in  the   process that we have just  described.  Next, we perform an additional coupling of this process (and hence of our modified process) with a $d$-dimensional FPP process.

Consider a $d$-dimensional lattice $\mathbb{Z}^d$, where $d=\log_2 \size + 2$ is assumed to be an integer. Each node of $\mathbb{Z}^d$ is represented by a vector of integer coordinates $v = (v_1,...,v_d)$. We   construct a mapping  $(Z',\mathbb{Z}^2) \rightarrow \mathbb{Z}^d$ as follows.
We   identify each node of $Z'$ by the first two coordinates   $(v_1,v_2)$ of   $\mathbb{Z}^d$. For every    $(v_1,v_2) \in Z'$   we  then assign the $N$ nodes of $\mathbb{Z}^2$ of the corresponding $w$-neighborhood, to    points of $\mathbb{Z}^d$  by letting the remaining $d-2$ coordinates be in  $\{0,1\}^{(d-2)}$. 
%
%
We   consider  a site FPP process on $\mathbb{Z}^d$ where the passage times for each image node on $\mathbb{Z}^d$ from $\mathbb{Z}^2$   are chosen to be equal to the flipping times  in our modified process.  The passage times for  the rest of the nodes 
are chosen to be  i.i.d  with distribution $F$. We call this our \textit{first intermediate FPP process}. By coupling, it should be clear that 
 if there is a self-avoiding path from the origin with at least a given number of steps and at most a given passage time in our modified model, then there is such a path also in   the auxiliary process described in the beginning of this section, and   in turn in the first intermediate FPP process. It follows that we can  bound the probability of occurrence of such an event in our modified model by the probability of occurrence of the corresponding event in the first intermediate FPP process. We state this bound rigorously in the following.

\begin{definition}[Path and First Passage Time  for the First Intermediate FPP]
Let a path $r'$ be a set of nodes in the first intermediate FPP setting such that the sequence $v_1,v_2,\ldots,v_k$ of passages with passage times $t_1,...,t_k$ is possible. Let
\begin{align*}
T'(r') = \sum_{i=1}^{k} t_i.
\end{align*}

Let $\mathcal{N}$ be a neighborhood. 
We define the \textit{first passage time} from the origin to $\mathcal{N}$ in the first intermediate FPP process as
\begin{align*}
T'(\mathbf{0},\mathcal{N})  := \inf_{r'\in \mathcal{P}'} \{T'(r') \},
\end{align*}
where $\mathcal{P}'$ is the set of possible sequences of passages starting from the origin and ending in $\mathcal{N}$.
%
Finally, we write $\xi_i$ for the $i$'th coordinate vector and define the passage time from $0$ to $n\xi_1$ as
\begin{align*}
a'_{0,n} := T'(\mathbf{0^*},n\xi_1).
\end{align*}
\end{definition}

Due to the coupling described above,  we have the following lemma which is used in the proof of the concentration bound.
\begin{lemma*}\label{Lemma:firstFPP}
Let $c,c'>0$ be constants. We have
\begin{align}\nonumber
\begin{split}
P\{\exists \mbox{ self-avoiding path $r$ starting at $0$} \mbox{ of at least $\lceil cn \rceil$ steps} \\ \mbox{ and with $T(r) < c'n$}\} \\
\le P\{\exists \mbox{ self-avoiding path $r'$ starting at $0$ of at least $\lceil cn \rceil$ steps} \\ \mbox{  and with $T'(r') < c'n$}\}.
\end{split}
\end{align}
\end{lemma*}
%


\


Now we recall a proposition from \cite{kesten1986aspects}. Let $0<p_c<1$ be the critical probability of standard site percolation on the lattice. 

\begin{proposition*}[Kesten \cite{kesten1986aspects}, Proposition~5.8] \label{Prop:Kesten5.8}
If 
\begin{align*} 
F(0) = P\{t(v)=0\} < p_c,
\end{align*}
 then there exist constants $0 < B,C,D< \infty$, depending on $d$ and $F$ only, such that
\begin{align*}
\begin{split}
P\{ \exists \text{ self-avoiding path } r' \text{ from the origin which contains at least } n \\ \text{ edges  and has } T'(r') < Bn \} \\
\le C e^{-Dn}.
\end{split}
\end{align*}
\end{proposition*}

We now slightly modify the above proposition to account for $d=\log_2N+2$. This proposition is  used in the proof of Theorem~\ref{Thrm:Theorem_Concentration}. The proof follows easily from the proof of the original statement and is omitted.

\begin{proposition*}[Kesten -- Modified] \label{Prop:mod_kesten}
Let $c>0$ be a constant, $n\ge 2^{c\size}$, and $d=\log_2 \size + 2$ be an integer.
If (\ref{eq:1.1}) holds then for sufficiently large $\size$, there exist constants $0 < B,C,D < \infty, 1\le E < \infty$, depending on $F$ only, such that
\begin{align*}
P\{ \exists \mbox{ self-avoiding path } r' \mbox{ from the origin which contains at least } n \\ \mbox{ nodes and has } T'(r') < Bn \} \\
 \le C e^{-Dn/{(\log n)^E}}.
\end{align*}
\end{proposition*}

\subsection{Second intermediate FPP process}
Consider the integer lattice $\mathbb{Z}^2$ and re-normalize it by partitioning the lattice into $w$-blocks such that we have a $w$-block centered at the origin and these blocks form a new lattice $L'$. We assume that each node of $L'$ is connected to its horizontal and vertical neighbors. 
Consider a first passage percolation process on this lattice that is coupled with our modified process by assigning the same random flipping times $t_i$'s of our modified process to the nodes of $\mathbb{Z}^2$. 
We define the passage time of each node on $L'$ as the sum of all the $t_i$'s of the particles of its corresponding block of nodes in $\mathbb{Z}^2$.
We point out that, in our modified process, when we assume that $\mathcal{N}_v$ is a region of expansion, once we have a monochromatic block, it can make a neighboring block monochromatic by flipping its particles one by one such that the time of this event is no larger than the sum of the flipping times of all the particles in this block. This fact implies that our second intermediate FPP provides an upper bound on the passage time from one node to the other on the modified process. These statements are made rigorous in the following.



\begin{definition}[Path and First Passage Time for the Second Intermediate FPP]Let a path $r''$ be a set of nodes $v_1,v_2,\ldots,v_k$ in $\mathbb{Z}^2$ with flipping times $t_1, t_2,\ldots,t_k$ such that they correspond to a possible passage in the second intermediate FPP process on $L'$ as described above. Let
\begin{align*}
T''(r'') = \sum_{i=1}^{k} t_i.
\end{align*}
Let $\mathcal{N}$ be a neighborhood. 
We define the \textit{first passage time} from the origin to $\mathcal{N}$ as
\begin{align*}
T''(\mathbf{0},\mathcal{N})  := \inf_{r''\in \mathcal{P}''} \{T''(r'') \},
\end{align*}
where $\mathcal{P}''$ is the set of possible sequences of passages starting from the origin and ending in $\mathcal{N}$.
Finally, we write $\xi_i$ for the $i$'th coordinate vector and define the passage time from $0$ to $n\xi_1$ as
\begin{align*}
a''_{0,n} := T''(\mathbf{0^*},n\xi_1).
\end{align*}

\end{definition}
Due to the coupling described above, we have the following lemma.
\begin{lemma*}\label{Lemma:secondFPP}
For any $a,y>0$, we have
\begin{align*}
P\{a_{0,n} \ge ayn\} \le P\{a''_{0,n} \ge ayn\}.
\end{align*} 
\end{lemma*}
\subsection{The Bound} \label{subsec:the_bound}
\

The following theorem is a restatement of Theorem~3 in \cite{kesten1993speed} which plays a critical role in the proof of our concentration bound.

\begin{theorem*}[Kesten \cite{kesten1993speed}]\label{Thrm:Kesten_Azuma}
Let $\{\mathscr{F}_k\}_{0\le k\le V}$ be an increasing family of $\sigma$-fields of measurable sets and let $\{U_k\}_{0\le k \le V}$ be a family of positive random variables that are $\mathscr{F}_V$-measurable. (We do not assume that $U_k$ is $\mathscr{F}_k$-measurable.) Let $\{Z_k\}_{0\le k \le V}$ be a martingale with respect to $\{\mathscr{F}_k\}_{0\le k\le V}$. (We allow $V=\infty$, in which case $\mathscr{F}_V = \vee_0^\infty \mathscr{F}_k$ and we merely assume that $\{Z_k\}_{0\le k < \infty}$ is a martingale.) Assume that the increments $\Delta_k := Z_k-Z_{k-1}$ satisfy
\begin{align}\label{eq:1.24}
|\Delta_k| \le c \mbox{ for some constant } c,
\end{align} 
and
\begin{align}\label{eq:1.25}
E\{\Delta^2_k | \mathscr{F}_{k-1} \} \le E\{ U_k | \mathscr{F}_{k-1} \}.
\end{align}

Assume further that for some constants $0 < C_1, C_2<\infty$ and $x_0$ with 
\begin{align}\label{eq:1.26}
x_0 \ge e^2c^2,
\end{align}
we have
\begin{align}\label{eq:1.27}
P\left\{ \sum_{k=1}^V U_k > x \right\} \le C_1e^{-C_2x} \mbox{ when } x \ge x_0.
\end{align}

Then, in the case where $V=\infty$, $Z_V = \lim_{k\rightarrow \infty} Z_k$ exists w.p.1. Moreover, irrespective of the value of $V$, there exist universal constants $0 < C_3,C_4 < \infty$ that do not depend on $V, C_1, C_2, c \mbox { and } x_0$, nor on the distribution of $\{Z_k\}$ and $\{ U_k \}$, such that 
\begin{align}\label{eq:1.28}
P\{Z_V-Z_0 \ge x \} \le C_3\left( 1 + C_1 + \frac{C_1}{C_2x_0} \right) \times \exp\left(-C_4 \frac{x}{x_0^{1/2} + C_2^{-1/3}x^{1/3}} \right).
\end{align}
In particular, for $x\le C_2x_0^{3/2}$,
\begin{align}\label{eq:1.29} 
P\{Z_V - Z_0 \ge x \} \le C_3 \left( 1 + C_1 + \frac{C_1}{C_2x_0} \right)\exp \left( - \frac{C_4x}{2\sqrt{x_0}} \right).
\end{align} 
\end{theorem*}

With a bit of abuse of notation, let 
\begin{align*}
a_{0,n} := T(\mathbf{0^*},n\xi_1).
\end{align*} 

\begin{proof}[Proof of Theorem~\ref{Thrm:Theorem_Concentration}] 
We order the vertices of $G_w$ in some arbitrary way, $v_1,v_2,\ldots$ which remains fixed throughout. 
Without loss of generality for the modified model, we work with a slightly simpler probability  space which does not take into account the initial configuration of the particles $\{{\theta},\bar{\theta}\}^{\mathbb{Z}^2}$, remembering that it is not possible to have multiple flipping times for each particle. This probability space  is
\begin{align*}
\Omega =  \mathbb{R}_+ \times \mathbb{R}_+ \times \ldots, \ \ \ \mathbb{R}_+ = [0,\infty),
\end{align*}
and a generic point of $\Omega$ is denoted by $\omega = (\omega_1,\omega_2,\ldots)$. In the configuration $\omega$, the flipping time of $v_i$ is 
\begin{align*}
t(v_i) = t(v_i,\omega) = \omega_i.
\end{align*}

When it is necessary to indicate the dependence on $\omega$ we write $a_{0,k}(\omega)$ instead of $a_{0,k}$, $T(r,\omega)$ instead of $T(r)$, etc. We shall use the following $\sigma$-fields of subsets of $\Omega$:
\begin{align*}
\mathscr{F}_0 = \mbox{ the trivial } \sigma\mbox{-field } = \{\emptyset , \Omega\}, \\
\mathscr{F}_k = \sigma\mbox{-field generated by } \omega_1,\ldots,\omega_k, \ \ k \ge 1. 
\end{align*}
The martingale representation of $a_{0,n}-\mathbb{E}[a_{0,n}]$ is 
\begin{align}\label{eq:2.1}
a_{0,n} - \mathbb{E}[a_{0,n}] = \sum_{k=1}^\infty \left( \mathbb{E}\{ a_{0,n}|\mathscr{F}_k \} - \mathbb{E}\{a_{0,n}| \mathscr{F}_{k-1}\} \right).
\end{align}
This representation is valid because $Z_0:=0$ and
\begin{align}\label{eq:2.2}
Z_l &:= \sum_{k=1}^l \left(\mathbb{E}\{a_{0,n}|\mathscr{F}_k\} - \mathbb{E}\{a_{0,n}|\mathscr{F}_{k-1}\} \right) \\
&= \mathbb{E}\{a_{0,n}|\mathscr{F}_l\} - \mathbb{E}[a_{0,n}], \ \ l\ge 1,
\end{align}
defines an $\{\mathscr{F}_l\}$-martingale that converges w.p.1 to  $a_{0,n} - \mathbb{E}[a_{0,n}]$. The increments of $\{Z_l \}$ are denoted by 
\begin{align}\label{eq:2.3}
\Delta_k = \Delta_{k,n}(\omega) = \mathbb{E}\{ a_{0,n} | \mathscr{F}_k \} - \mathbb{E}\{a_{0,n}|\mathscr{F}_{k-1}\}.
\end{align}

The principal step is to estimate 
\begin{align*}
\mathbb{E}\{\Delta^2_k | \mathscr{F}_{k-1}\}.
\end{align*}
To this end, we write 
\begin{align*}
a_{0,n}(\omega) = f(t(v_1,\omega),t(v_2,\omega),\ldots) = f(\omega_1,\omega_2,\ldots)
\end{align*}
for some Borel function $f : \Omega \rightarrow \mathbb{R}_+$. Also, the following notation is useful. If $\omega = (\omega_1,\omega_2,\ldots)$ and $\sigma = (\sigma_1,\sigma_2,\ldots)$ are points of $\Omega$, then
\begin{align*}
[\omega,\sigma]_k = (\omega_1,\ldots,\omega_k,\sigma_{k+1},\ldots)
\end{align*}
is the point that agrees with $\omega$ and $\sigma$ on the first $k$ coordinates and the coordinates after $k$, respectively. $\nu_{k+1}$ will be the product measure 
\begin{align*}
\nu_{k+1} =  \Pi_{k+1}^{\infty} F_i
\end{align*}
on the obvious $\sigma$-field in
\begin{align*}
\Omega_{k+1} =  R_{k+1} \times R_{k+2} \times \ldots 
\end{align*}
when each $R_i = \mathbb{R}_+$ and $F_i = F$. We can think of $\Omega$ as $R_1 \times \ldots \times R_k \times \Omega_{k+1}$ and if $g$ is a function from $\Omega \rightarrow \mathbb{R}$, then if we fix $\sigma_1,\ldots,\sigma_k, g(\sigma)$ can be viewed as a function of $\sigma_{k+1},\sigma_{k+2},\ldots$; that is, as a function on $\Omega_{k+1}$. Correspondingly, 
\begin{align*}
\int _{\Omega_{k+1}}\nu_{k+1}(d\sigma)g(\sigma) := \int {\Pi_{k+1}^{\infty}} F(d\sigma_i)g(\sigma_1,\ldots,\sigma_k,\sigma_{k+1},\ldots)
\end{align*} 
is the integral over all coordinates $\sigma_i$, with $i\ge k+1$, and is a function of $\sigma_1,\ldots,\sigma_k$.
By the independence of the $t(v_i,\omega)=\omega_i$, $i\ge 1$, we have 
\begin{align}\label{eq:2.6}
\mathbb{E}\{a_{0,n} | \mathscr{F}_k\}(\omega) = \int_{\Omega_{k+1}}\nu_{k+1}(d\sigma)f([\omega,\sigma]_k).
\end{align}
This is a function of $t(v_i,\omega)=\omega_i$, $1\le i \le k$, only. It also equals
\begin{align} \label{eq:2.7}
\int_{\Omega_k}\nu_k(d\sigma)f([\omega,\sigma]_k),
\end{align}
because $[\omega,\sigma]_k$ does not involve $\sigma_k$ and the integration over $\sigma_k$ has no effect. Using (\ref{eq:2.7}) for $\mathbb{E}\{a_{0.n}|\mathscr{F}_k\}$ and (\ref{eq:2.6}) with $k$ replaced by $(k-1)$ for $\mathbb{E}\{a_{0,n}|\mathscr{F}_{k-1}\}$, we find
\begin{align}\label{eq:2.8}
\Delta_k = \int_{\Omega_k} \nu_k(d\sigma)\{f[\omega,\sigma]_k-f([\omega,\sigma]_{k-1})\}.
\end{align}
Our task now is to estimate 
\begin{align*}
g_k(\omega,\sigma) := |f([\omega,\sigma]_k)-f([\omega,\sigma]_{k-1})|.
\end{align*}
Note that
\begin{align}\label{eq:2.10}
t(v_i,[\omega,\sigma]_k) = t(v_i,[\omega,\sigma]_{k-1})  = 
\begin{cases}
t(v_i,\omega),  \ \ \mbox{if } i\le k-1,  \\
t(v_i,\sigma), \ \ \mbox{if } i\ge k+1.
\end{cases}
\end{align}
Only for $i=k$ do we obtain different values for the flipping time of $v_i$ in the two configurations $[\omega,\sigma]_k$ and $[\omega,\sigma]_{k-1}$:
\begin{align*}
t(v_k,[\omega,\sigma]_k) = t(v_k,\omega), \ \ \ t(v_k,[\omega,\sigma]_{k-1}) = t(v_k,\sigma).
\end{align*}
We claim that this implies
\begin{align}\label{eq:2.12}
g_k(\omega,\sigma) \le |t(v_k,\omega)-t(v_k,\sigma)|.
\end{align}
Indeed, for any path $r$,
\begin{align}\label{eq:2.13}
|T(r,[\omega,\sigma]_k)-T(r,[\omega,\sigma]_{k-1})| &\le \sum_{v \in r} |t(v,[\omega,\sigma]_k)-t(v,[\omega,\sigma]_{k-1})| \\
&\le |t(v_k,\omega) - t(v_k,\sigma)|. \nonumber
\end{align}
Therefore, the same estimate holds for
\begin{align*}
|a_{0,n}([\omega,\sigma]_k)-a_{0,n}([\omega,\sigma]_{k-1})| = |\inf_r T(r,[\omega,\sigma]_k)-\inf_r T(r,[\omega,\sigma]_{k-1})|.
\end{align*}

This proves (\ref{eq:2.12}). Let $\pi_n(\omega)$ be the
optimal path from $0$ to $n \xi_1$ in the
configuration $\omega$; that is, $\pi_n(\omega)$ is a path from $0$ to $n\xi_1$ with
\begin{align}\label{eq:2.14}
a_{0,n}(\omega) = T(\pi_n(\omega),\omega).
\end{align}
In our modified process, such a path always exists. There could, however, be several paths with this property. To define $\pi_n(\omega)$ uniquely in case of ties, we order all paths from $0$ to $n\xi_1$ in some arbitrary way, and take for $\pi_n(\omega)$ the first path in this ordering that satisfies (\ref{eq:2.14}). We write $v \in \pi$ to denote that $v$ is a node in the path $\pi$. Then, if
\begin{align}\label{eq:2.15}
v_k \not \in \pi_n([\omega,\sigma]_k),
\end{align}
(\ref{eq:2.10}) and (\ref{eq:2.13}) show that 
\begin{align*}
T(\pi_n([\omega,\sigma]_k),[\omega,\sigma]_k) = T(\pi_n([\omega,\sigma]_k),[\omega,\sigma]_{k-1}).  
\end{align*}
Thus, under (\ref{eq:2.15}),
\begin{align*} 
a_{0,n}([\omega,\sigma]_{k-1}) &= \inf \{T(r,[\omega,\sigma]_{k-1}):r \mbox{ a path from } 0 \mbox{ to } n\xi_1 \} \\
& \le T(\pi_n([\omega,\sigma]_{k}),[\omega,\sigma]_{k-1}) \\
& = T(\pi_n([\omega,\sigma]_{k}),[\omega,\sigma]_{k}) \\
&= a_{0,n}([\omega,\sigma]_{k}).
\end{align*}
Similarly, if
\begin{align}\label{eq:2.16}
v_k \not \in \pi_n([\omega,\sigma]_{k-1}),
\end{align}
then
\begin{align*}
a_{0,n}([\omega,\sigma]_{k}) \le a_{0,n}([\omega,\sigma]_{k-1}). 
\end{align*}
It follows that $g_k(\omega,\sigma)=0$ if (\ref{eq:2.15}) and (\ref{eq:2.16}) both hold, and by virtue of (\ref{eq:2.12}),
\begin{align}\label{eq:2.17}
g_k(\omega,\sigma) \le |t(v_k,\omega)-t(v_k,\sigma)| \times 1_{\{v_k \in \pi_n([\omega,\sigma]_{k-1}) \text{ or } v_k \in \pi_n([\omega,\sigma]_k)\}}.
\end{align}
This is then an estimate for $g_k$. Let $1_k(\omega, \sigma)$ denote the indicator function in the right-hand side of the above inequality. Now using Schwarz's inequality and (\ref{eq:2.8}) we have
\begin{align}\label{eq:2.18}
\begin{split} 
\mathbb{E}\left[ \Delta^2_k | \mathscr{F}_{k-1} \right] &\le \mathbb{E}\left[ \left( \int_{\Omega_k} \nu_k(d\sigma)g_k(\omega,\sigma) \right)^2 \given[\Big]   \mathscr{F}_{k-1} \right] \\
&\le \mathbb{E}\left[ \left( \int_{\Omega_k} \nu_k(d\sigma)|t(v_k,\omega)-t(v_k,\sigma)| 1_k(\omega,\sigma) \right)^2 \given[\Big] \mathscr{F}_{k-1} \right] \\
&\le \mathbb{E}\left[ \int_{\Omega_k} \nu_k(d\sigma)|t(v_k,\omega)-t(v_k,\sigma)|^2 1_k(\omega,\sigma)   \right. \\ 
& \ \ \ \ \  \ \ \ \ \  \ \ \ \ \ \ \ \times \left. \int_{\Omega_k} \nu_k(d\sigma) 1_k(\omega,\sigma) \given[\Big] \mathscr{F}_{k-1} \right] \\
&\le \mathbb{E}\left[ \int_{\Omega_k} \nu_k(d\sigma)|t(v_k,\omega)-t(v_k,\sigma)|^2 1_k(\omega,\sigma)  \given[\Big] \mathscr{F}_{k-1} \right].
\end{split}
\end{align}
Now
\begin{align*}
 \int_{\Omega_k} \nu_k(d\sigma)|t(v_k,\omega)-t(v_k,\sigma)|^2 1_k(\omega,\sigma) ,
\end{align*}
is a function of $\omega_1,\ldots,\omega_k$ only; the $\sigma$-variables all have been integrated out. Similar to (\ref{eq:2.6}) we have
\begin{align}\label{eq:2.19}
\begin{split}
\mathbb{E}&\left[ \int_{\Omega_k} \nu_k(d\sigma)|t(v_k,\omega)-t(v_k,\sigma)|^2 1_k(\omega,\sigma)  \given[\Big] \mathscr{F}_{k-1} \right] \\
&= \int F(d\omega_k) \int_{\Omega_k} \nu_k(d\sigma) |t(v_k,\omega)-t(v_k,\sigma)|^2 \times 1_k(\omega,\sigma) \\
&= \int F(d\omega_k) \int F(d\sigma_k) \int_{\Omega_{k+1}}\nu_{k+1}(d\sigma)|t(v_k,\omega)-t(v_k,\sigma)|^2 \times 1_k(\omega,\sigma).
\end{split}
\end{align}
where we have used the fact that $\nu_k$ can be written as the product measure $F\times \nu_{k+1}$ on $R_k \times \Omega_{k+1} = \mathbb{R}_+\times \Omega_{k+1}$. Let us write
\begin{align*}
J_k(\omega) = 1_{\{v_k \in \pi_n(\omega)\}}.
\end{align*}
Then
\begin{align}\label{eq:2.20}
\begin{split}
1_k(\omega,\sigma) &= J_k([\omega,\sigma]_{k-1}) \vee J_k([\omega,\sigma]_k) \\
&= J_k(\omega_1,\ldots,\omega_{k-1},\sigma_k,\sigma_{k+1},\ldots) \vee J_k(\omega_1,\ldots,\omega_{k-1},\omega_k,\sigma_{k+1},\ldots).
\end{split}
\end{align}

Recall that $t(v_k,\omega) = \omega_k$, $t(v_k,\sigma) = \sigma_k$, so that
\begin{align*}
\begin{split}
|t(v_k,\omega)&-t(v_k,\sigma)|^2 1_k(\omega,\sigma) \\
&\le |\omega_k-\sigma_k|^2\left\{J_k(\omega_1,\ldots,\omega_{k-1},\sigma_k,\sigma_{k+1},\ldots) \right. \\
& \ \ \ \ \ \ \ \ \ \ \ \ \ \ \ \ \ \ \ \ \ \ \  \left. \vee J_k(\omega_1,\ldots,\omega_{k-1},\omega_k,\omega_k,\sigma_{k+1},\ldots)\right\}.
\end{split}
\end{align*}
Clearly the right-hand side is symmetric in $\omega_k$ and $\sigma_k$ for fixed 
\begin{align*}
\omega_1,\ldots,\omega_{k-1}, \sigma_{k+1},\sigma_{k+2},\ldots ,
\end{align*}
 (in fact, this is true also for the left-hand-side). It is also clear that on $\{\sigma_k \le \omega_k \}$ or on $\{t(v_k,\sigma) \le t(v_k,\omega)\}$ we have 
\begin{align}\label{eq:2.21}
|t(v_k,\omega)-t(v_k,\sigma)| \le t(v_k,\omega),
\end{align}
\begin{align}\label{eq:2.22}
J_k(\omega_1,\ldots,\omega_{k-1},\sigma_k,\sigma_{k+1},\ldots) \vee J_k(\omega_1,\ldots,\omega_{k-1},\omega_k,\sigma_{k+1},\ldots) \\
= J_k(\omega_1,\ldots,\omega_{k-1},\sigma_k,\sigma_{k+1},\ldots) = J_k([\omega,\sigma]_{k-1}). \nonumber
\end{align}
(\ref{eq:2.22}) simply says that if $v_k$ belongs to the optimal path in configuration $[\omega,\sigma]_{k-1}$ or in configuration $[\omega,\sigma]_k$, then it will belong to the optimal path in the configuration that gives the lower value to $t(v_k)$. Substituting (\ref{eq:2.19})-(\ref{eq:2.22}) into (\ref{eq:2.18}) we find 
\begin{align}\label{eq:2.23}
\mathbb{E}[\Delta^2_k|&\mathscr{F}_{k-1}] \\
  &\le 2\int_{\Omega_{k+1}}\nu_{k+1}(d\sigma)\int\int_{\sigma_k\le \omega_k}F(d\omega_k)F(d\sigma_k)t^2(v_k,\omega)J_k([\omega,\sigma]_{k-1}) \nonumber \\
& \le  2\int F(d\omega_k)t^2(v_k,\omega)\int F(d\sigma_k)\int_{\Omega_{k+1}}\nu_{k+1}(d\sigma)J_k([\omega,\sigma]_{k-1}) \nonumber \\
& = \left(2\int x^2 dF(x)\right)P\{v_k \in \pi_n(\omega)|\mathscr{F}_{k-1}\}.\nonumber
\end{align}



Let $|\pi|$ denote the number of vertices in $\pi$. For any $a>0$, $y>0$,
\begin{align}\label{eq:2.26}
\begin{split}
P\{|\pi_n| \ge yn\} \le & P\{a_{0,n} \ge ayn\} \\
 &+ P\{\exists \text{ self-avoiding path $r$ starting at $0$} \\ 
 &\text{ of at least $yn$ steps and with $T(r) < ayn$}\}.
 \end{split}
\end{align}
 
Now we want to use our intermediate FPP processes to bound the probabilities on the right hand side of (\ref{eq:2.26}). Let us decorate the variables corresponding to our first and second intermediate FPP processes by a prime and a double prime respectively. We are going to find upper bounds for the right hand side terms of (\ref{eq:2.26}) using these intermediate processes. Let $r''_n$ be the optimal path between the origin and $n\xi_1$ in our second  intermediate process.  Using Lemma~\ref{Lemma:secondFPP} we have
\begin{align}\label{eq:2.27}
P\{a_{0,n} \ge ayn\} \le P\{a''_{0,n} \ge ayn\} = P\{T''(r''_n) \ge ayn\} \le P\left\{\sum_{i=1}^{w'n} t_i \ge ayn\right\},
\end{align}
where 
$w' = O(w)$ adjusts the number of particles for the renormalized graph (our second intermediate process).
Using Lemma~\ref{Lemma:firstFPP}, we also have
\begin{align}\label{eq:2.26'}
\begin{split}
P\{\exists \mbox{ self-avoiding path $r$ starting at $0$} \mbox{ of at least $yn$ steps} \\ \mbox{ and with $T(r) < ayn$}\} \\
\le P\{\exists \mbox{ self-avoiding path $r'$ starting at $0$ of at least $yn$ steps} \\ \mbox{  and with $T'(r') < ayn$}\}.
\end{split}
\end{align}
For a suitable $a$, using Proposition~\ref{Prop:mod_kesten}, and for sufficiently large $\size$, we have
\begin{align*}
\begin{split}
P\{\exists \text{ self-avoiding path $r'$ starting at $0$ of at least $yn$ steps}  \\ \mbox{ and with $T'(r') < ayn$}\} \\
\le C_2 \exp(-C_3yn/(\log n)^{C_4}). 
\end{split}
\end{align*}
Hence, we have
\begin{align}\label{eq:2.26_new}
\begin{split}
P\{|\pi_n| \ge yn\} \le P\left\{\sum_{i=1}^{w'n} t_i \ge ayn\right\} + C_2 \exp(-C_3yn/(\log n)^{C_4}).
 \end{split}
\end{align}

We  show that
\begin{align} \label{eq:1.15*}
P\left\{\left|a_{0,n}-\mathbb{E}[a_{0,n}] \right| \ge x    \right\} \le C_3e^{-C_4x/\sqrt{n\log n}} \ \ for \ \ x\le n,
\end{align} 
when $n\ge 2^{c\size}$ for some constant $c>0$ and $\size$ sufficiently large. Since our proof does not rely on any special properties for the direcitons along the coordinate axes, we can conclude that for any arbitrary direction one can obtain a similar result. In particular (\ref{eq:1.15}) holds.


We want to use Theorem~\ref{Thrm:Kesten_Azuma}, for which we need a truncation argument. Let $n$ be fixed. Define
\begin{align}\label{eq:2.28}
\hat{t}(v_i) = t(v_i)\wedge \frac{4d}{\gamma}\log n ,
\end{align}
with $\gamma$ as in (\ref{eq:1.14}). Passage times and related quantities, when defined in terms of the $\hat{t}$ instead of the $t$, will be denoted by the symbols decorated with a caret. For example, if $r=(v_1,\ldots,v_k)$, then
\begin{align*}
&\hat{T}(r) = \sum_{i=1}^k \hat{t}(v_i); \\
&\hat{a}_{0,n} = \inf\{\hat{T}(r): r \text{ a path from $0$ to $n\xi$}\}, \\
&\hat{\pi}_n = \mbox{optimal path for $\hat{a}_{0,n}$}.
\end{align*}
\begin{lemma*}\label{Lemma:lemma1}
If (\ref{eq:1.1}) and (\ref{eq:1.14}) hold, then there exist constants $1/2 < C_1 < \infty$, $0 < C_i < \infty$ for $i>1$, such that
\begin{align}\label{eq:2.29}
P\{ \hat{\pi}_n \not \subseteq [-n(\log n)^{C_1},n(\log n)^{C_1}]^d \} \le 3e^{-C_2n},
\end{align}
\begin{align}\label{eq:2.30}
P\{ |a_{0,n}-\hat{a}_{0,n}| \ge x \} \le 3e^{-C_2n} + C_3e^{-(\gamma/2)x}, \ x \ge 0,
\end{align}
and
\begin{align}\label{eq:2.31}
|\mathbb{E}[a_{0,n}]-\mathbb{E}[\hat{a}_{0,n}]| \le C_4.
\end{align}
\end{lemma*}
\begin{proof}
The probability in the left-hand side of (\ref{eq:2.29}) is bounded by
\begin{align}\label{eq:2.32}
\begin{split}
P\{|\hat{\pi}_n| \ge n(\log n)^{C_1}\} \le P\{\hat{a}_{0,n} \ge an(\log n)^{C_1}\} \\
+ P\{\exists \text{ self-avoiding path $r$ starting at $0$} \\ 
\text{ of at least $n(\log n)^{C_1}$ steps and with $\hat{T}(r) < an(\log n)^{C_1}$}\} \\
 \le P\left\{\sum_{i=1}^{w'n} t_i \ge an(\log n)^{C_1}  \right\} + C_5 \exp(-C_6n(\log n)^{C_1}/{(\log n)^{C_7}}),
\end{split}
\end{align}
similar to (\ref{eq:2.26_new}). Note that the last inequality is true for some constants $C_5,C_6$ that are independent of $n$, although the distribution of $\hat{t}$ does depend on $n$. This follows because for any constant $C$ and  for large enough $n$, we have $\hat{t}(v) \ge t(v)\wedge C$. Thus also $\hat{T}(t) \ge \sum_{v \in r} \{t(v)\wedge C\}$, and it suffices to apply Proposition~\ref{Prop:mod_kesten} and choose $C_1$ to be equal to $C_7$. 

Now we use (\ref{eq:1.14}) for the following standard large deviation estimate:

\begin{align}\label{eq:2.33}
\begin{split}
P\left\{ \sum_{i=1}^{w'n} t_i \ge an(\log n)^{C_1} \right\} \le e^{-\gamma a n(\log n)^{C_1}} \left(E[e^{\gamma t_1}] \right)^{w'n}.
\end{split}
\end{align}

Now, using the fact that $n \ge 2^{c\size}$,  for sufficiently large $\size$, and $C_1 > 1/2$, the above expression is bounded above by $\exp(-n)$. Now, the inequality in (\ref{eq:2.29}) follows. To prove (\ref{eq:2.30}) and (\ref{eq:2.31}) we note that
\begin{align} \label{eq:2.34}
\begin{split}
0 &\le a_{0,n} - \hat{a}_{0,n} \le T(\hat{\pi}_n) - \hat{T}(\hat{\pi}_n) \\
&= \sum_{v\in \hat{\pi}_n} \left\{t(v) - \hat{t}(v) \right\} \le \sum_{v\in \hat{\pi}_n} t(v)1_{\{t(v) > \frac{4d}{\gamma}\log n\}}.
\end{split}
\end{align}
If $\hat{\pi}_n \subseteq [-n(\log n)^{C_1}, n(\log n)^{C_1}]^d$, then the last term of (\ref{eq:2.34}) is at most 
\begin{align*}
\sum_{v\in [-n(\log n)^{C_1},n(\log n)^{C_1}]^d} t(v) 1_{\{t(v)>\frac{4d}{\gamma}\log n\}}.
\end{align*}
Hence,
\begin{align} \label{eq:2.35}
\begin{split}
P\{|a_{0,n}-\hat{a}_{0,n}| \ge x\} 
\le \  &P\{\hat{\pi}_n \not \subseteq [-n(\log n)^{C_1},n(\log n)^{C_1}]^d\} \\
     &+ P\left\{\sum_{i=1}^{M'} t_i 1_{\{t_i > \frac{4d}{\gamma}\log n\}} \ge x \right\},
\end{split}
\end{align}
where 
\begin{align} \label{eq:2.36}
M' = \mbox{number of nodes in } [-C_1n\log n,C_1n\log n]^d \sim  d(2C_1n\log n)^d.
\end{align}
Therefore, we have
\begin{align} \label{eq:2.37}
\begin{split}
P&\left\{\sum_{i=1}^{M'}t_i 1_{\{t_i > \frac{4d}{\gamma}\log n\}} \ge x \right\} \\ 
&\le e^{-(\gamma/2)x}\left[1+\int_{y\ge 4(d/\gamma)\log n} (e^{(\gamma/2)y}-1)F(dy) \right]^{M'} \\
&\le e^{-(\gamma/2)x}\left[1+e^{-2d\log n}\int e^{\gamma y}F(dy) \right]^{M'} \\
& \le \exp \left\{-(\gamma/2)x + M'n^{-2d}\int e^{\gamma y}F(dy) \right\}\\
& \le C_3e^{-(\gamma/2)x},
\end{split}
\end{align}
we now have that (\ref{eq:2.30}) follows from (\ref{eq:2.35}), (\ref{eq:2.29}) and (\ref{eq:2.37}); and
(\ref{eq:2.31}) follows from (\ref{eq:2.30})  and by the additional estimates,
\begin{align*}
0 \le \hat{a}_{0,n} \le a_{0,n}. 
\end{align*}
We also have that  there exists $y_0$ such that for  all $y \ge y_0$, 
\begin{align}\label{eq:2.38}
\begin{split}
P\{|a_{0,n}-\hat{a}_{0,n}| \ge yn\log n\} &\le P\{a_{0,n} \ge yn\log n\} \le P\{a''_{0,n} \ge yn\log n\} \\
&\le e^{-\gamma yn\log n} \left[\int e^{\gamma x}F(dx) \right]^{wn} \le e^{-(\gamma/2) yn\log n},
\end{split} 
\end{align}
where we have used the fact that replacing the $t_i$'s belonging to the optimal path with an arbitrary set of flipping times can only increase the probability of this event.
\end{proof}
Now we can prove (\ref{eq:1.15}). By Lemma~\ref{Lemma:lemma1} we have for $x\sqrt{n} \ge 2C_4$,
\begin{align*}
P\{ |a_{0,n} - \mathbb{E}[a_{0,n}]| \ge x\sqrt{n} \} \le P\left\{|a_{0,n} - \hat{a}_{0,n}| \ge \frac{x}{4}\sqrt{n} \right\} \\  +  P\left\{|\hat{a}_{0,n} - \mathbb{E}[\hat{a}_{0,n}]| \ge \frac{x}{4}\sqrt{n} \right\} \\
\le 3\exp(-C_2n) + C_3\exp\left(-\frac{\gamma}{8}x\sqrt{n}\right) \\ + P\left\{|\hat{a}_{0,n} - \mathbb{E}[\hat{a}_{0,n}]| \ge \frac{x}{4}\sqrt{n} \right\}.
\end{align*}
The first term in the right hand side is at most $3\exp(-C_2x)$ for $x\le n$. Hence if we prove the following, the proof of (\ref{eq:1.15}) is complete.
\begin{align}\label{eq:2.39}
P\left\{|\hat{a}_{0,n}-\mathbb{E}[\hat{a}_{0,n}]| \ge \frac{x}{4}\sqrt{n}\right\} \le C_7 \exp(-C_8x/\sqrt{\log n})  \ \ \mbox{for} \ \ x\le n.
\end{align}
The remaining part of the proof deduces (\ref{eq:2.39}) from Theorem~\ref{Thrm:Kesten_Azuma}. As in (\ref{eq:2.1})-(\ref{eq:2.3}),
\begin{align*}
\hat{a}_{0,n} - \mathbb{E}[\hat{a}_{0,n}] = \sum_{k=1}^\infty \hat{\Delta}_k
\end{align*}
with
\begin{align*}
\hat{\Delta}_k = \mathbb{E}[\hat{a}_{0,n}|\mathscr{F}_k] - \mathbb{E}[\hat{a}_{0,n}|\mathscr{F}_{k-1}].
\end{align*}
Moreover,
\begin{align*}
Z_0 = 0, \ \ Z_l = \sum_{k=1}^l \hat{\Delta}_k, \ \ l\ge 1,
\end{align*}
defines a martingale. We shall now verify the hypotheses of Theorem~\ref{Thrm:Kesten_Azuma} for this  martingale.
Note that replacing $t(v_i)$, $a_{0,n}$ and $\Delta_k$ by $\hat{t}(v_i)$, $\hat{a}_{0,n}$ and $\hat{\Delta}_k$ merely amounts to changing the distribution $F$ to 
\begin{align*}
\hat{F}(x) = F(x) \vee 1_{\{x\ge \frac{4d}{\gamma}\log n\}}.
\end{align*}
Therefore, by (\ref{eq:2.12}), (\ref{eq:2.8}) and the definition (\ref{eq:2.28}) of $\hat{t}$, we have
\begin{align*}
|\hat{\Delta}_k| \le 2\max (\mbox{supp }   \hat{F}) \le \frac{8d}{\gamma} \log n.
\end{align*}
 This corresponds to (\ref{eq:1.24}) by letting
\begin{align*}
c = \frac{8d}{\gamma}\log n.
\end{align*}
Furthermore, by (\ref{eq:2.23}), we have
\begin{align*}
\mathbb{E}[\hat{\Delta}_k^2 | \mathscr{F}_{k-1}] &\le \left(2\int x^2 \hat{F}(dx)\right) P\{v_k \in \hat{\pi}_n(\omega)| \mathscr{F}_{k-1}\} \\
&\le \left(2\int x^2 F(dx)\right) P \{v_k \in \hat{\pi}_n(\omega)| \mathscr{F}_{k-1}\}.
\end{align*}
Thus (\ref{eq:1.25}) holds with
\begin{align*}
U_k = D\hat{J}_k,
\end{align*}
where
\begin{align*}
D = 2\int x^2 F(dx), \ \ \ \hat{J}_k(w) = 1_{\{v_k \in \hat{\pi}_n(\omega)\}}.
\end{align*}

We next let  $C>0$ as in Proposition~\ref{Prop:mod_kesten} and
\begin{align}\label{eq:2.40}
x_0 = n\log n \frac{2D}{\gamma C} \log \left\{\int e^{\gamma u} F(du) \right\}.
\end{align}
Clearly this satisfies (\ref{eq:1.26}) for sufficiently large  $n$. Finally, we verify (\ref{eq:1.27}). We have
\begin{align}
\begin{split}
\sum_{k=1}^\infty U_k &= D\sum_{k=1}^\infty 1_{\{v_k \in \hat{\pi}(\omega)\}} \\
&= D|\hat{\pi}_n(\omega)| = D \times \mbox{length of } \hat{\pi}_n(\omega).
\end{split}
\end{align}
Moreover, as in (\ref{eq:2.32}) and (\ref{eq:2.33}), we have
\begin{align}\label{eq:2.42}
\begin{split}
P\{|\hat{\pi}_n(\omega)| \ge y\} \le P&\{\hat{a}_{0,n} \ge Cy\} \\
&+ P\{\exists \text{ self-avoiding path $r$ starting at $0$} \\ 
&\text{ of at least $y$ steps and with $\hat{T}(r) < Cy$}\} \\
& \le e^{-\gamma Cy}\left[\int e^{\gamma u}F(du) \right]^{wn} + C_9 \exp(-C_{10}y/{(\log n)^{C_{11}}}).
\end{split}
\end{align}
For $y\ge x_0/D$ and $x_0$ as in (\ref{eq:2.40}), the right-hand side of (\ref{eq:2.42}) is at most
\begin{align*}
e^{-(\gamma/2)Cy}+C_9e^{-C_{10}y/{(\log n)^{C_{11}}}}.
\end{align*}
Therefore, (\ref{eq:1.27}) holds with $C_1 = (1+C_9)$ and 
\begin{align*}
C_2 = \frac{\gamma C}{2D} \wedge \frac{C_{10}}{D(\log n)^{C_{11}}}.
\end{align*}
Thus, by (\ref{eq:1.29}) (applied to $\hat{a}_{0,n}$ and to $-\hat{a}_{0,n}$), for $n\ge 2^{c\size}$, $x\le n$ when $\size$ is sufficiently large, we have
\begin{align*}
P\left\{|\hat{a}_{0,n}- \mathbb{E}[\hat{a}_{0,n}]| \ge x \right\} &\le 2C_3(1+C_1)(1+\frac{C_1}{C_2 x_0})\exp(-\frac{C_4}{2}\frac{x}{\sqrt{x_0}}) \\
&\le \exp\left(-C_{12}\frac{x}{\sqrt{n\log n}}\right),
\end{align*}
which proves (\ref{eq:1.15}) for our process on $G_w$.
\end{proof}





\section{Proof of the Shape Theorem} \label{Sec:Shape_theorem}

\

 Throughout this section, we continue working with the modified model introduced in Section~\ref{Sec:Concentration}. The coupling used to obtain the final result for the original model occurs in the last steps of the proof in this section.

\

\noindent \textit{Notational Conventions.} Given a subset $A$ of $\mathbb{R}^2$, and $t \in \mathbb{R}$, we let
$tA = \{ ta \  | \  a \in A\}$, and $AB = \{a + b \  | \  a \in A, b \in B\}$. So in particular $A^n = \{ a_1 + \ldots + a_n \  |  \  a_i \in A \}$.

\ 

We introduce the following definition from Tessera~\cite{tessera2014speed} with some modifications. 

\begin{definition}[Strong Asymptotic Geodesicity  (SAG) for $\mathbb{E}{[}a{]}$] Let $Q: \mathbb{R}_+ \rightarrow \mathbb{R}_+$ be an increasing function such that 
\begin{align*}    
\lim_{\alpha \rightarrow \infty}Q(\alpha) = \infty.
\end{align*}
Let $c'_1\in(0,c_1)$ be a constant and let $\bar{A}(0,r) = \left\{x \  | \  \mathbb{E}[a_{0,x}]  \le r \right\}$ and $[\bar{A}]_t$ denote the $t$-neighborhood of the subset $\bar{A}$ with respect to $\mathbb{E}[a]$. $\mathbb{E}[a]$ is called $SAG(Q)$ when for all integers $m\ge 1$, and for all $x,y \in \mathcal{N}_v$ such that $ \mathbb{E}[a_{x,y}]/m \ge 2^{c'_1\size}$, there exists a sequence $x=x_0,\ldots.,x_m=y$ in $\mathcal{N}_v$ satisfying, for all $0\le i \le m-1$,
\begin{align}\label{eq:Tess_EQSAG1}  
\alpha\left(1 - \frac{1}{Q(\alpha)} \right) \le  \mathbb{E}[a_{x_i,x_{i+1}}] \le \alpha\left(1 + \frac{1}{Q(\alpha)} \right),
\end{align} 
where $\alpha	= \mathbb{E}[a_{x,y}]/m$; and for all $r \in [2^{c'_1\size},2^{c_2\size}] $, we have
\begin{align}\label{eq:Tess_EQSAG2} 
\bar{A}\left(0,\left(1+\frac{1}{Q(r)} \right)r \right) \subset [\bar{A}(0,r)]_{\frac{6r}{Q(r)}}.
\end{align}
\end{definition}

\begin{proposition*}\label{Prop:Tess1}
Let $\ti \in [2^{c_1\size},2^{c_2\size}] $ and $$A'_F(0,\ti) = \{ x \given a_{0,x} \le \ti  \}.$$
Almost surely, there exists $\size_0 \in \mathbb{N}$, such that for $\size \ge \size_0$,  there exist a norm $l_*$ on $\mathbb{R}^2$, and $c>0$, such that,
\begin{align}\label{eq:Tess1}
B_{l_*}(0,\ti-\size^c \ti^{1/2}\log^{3/2} \ti)\cap \mathbb{Z}^2 \subset A'_{F}(0,\ti)  \\ \subset B_{l_*}(0,\ti+\size^c \ti^{1/2}\log^{3/2} \ti). \nonumber
\end{align}
\end{proposition*}


The proof of this proposition, as outlined in Section~\ref{Sec:Model}, consists of two main parts. We now start by preparing the stage for the first part of the proof.

Let $\bar{d}_1$ and $\bar{d}_2$ be the metrics defined by the expected value of first passage times in the first and second intermediate FPPs when we choose the nodes to be on $G_w$ defined on $\mathbb{Z}^2$, i.e., 
$$\bar{d}_1 = \mathbb{E}[a'_{0,n}], \bar{d}_2 = \mathbb{E}[a''_{0,n}]  .$$
 Let $x,y$ denote two arbitrary nodes on $\mathcal{N}_v$ and $d_{l_2}$ denote the Euclidean metric. Since in the modified model we are assuming that $\mathcal{N}_v$ is a region of expansion and there are no affected nodes of any type on $\mathcal{N}_v$, it is easy to see that there exists $c,c' >0$ such that for sufficiently large $\size$ we can write
\begin{align}\label{eq:bilips}
 d_{l_2}(x,y)/\size^c \le \bar{d}_1(0,\xi_1\lfloor\|x-y\|_1/\size^{c'}\rfloor)  \le \mathbb{E}[a_{x,y}] \nonumber  \\
  \le \bar{d}_2(0,\xi_1\lceil\|x-y\|_1/\size^{c'}\rceil) \le \size^c d_{l_2}(x,y).
\end{align}

\begin{proposition*}[Tessera -- Modified]\label{Prop:Tess_fluc}
Let $2^{c'_1\size} \le r_w \le 2^{c_2\size}$ be a sequence where $c'_1, c_2$ are as defined before, and $\mathcal{N}_{r_w}$ be a sequence of neighborhoods with radii $r_w$ centered at the origin.
Almost surely, there exists $C>0$ and $w_0 \in \mathbb{N}$ (hence $\size_0 \in \mathbb{N}$)  such that for $w \ge w_0$ (hence $\size \ge \size_0$), we have
\begin{align}\label{eq:Tess_prop1}
\sup_{x,y\in \mathcal{N}_{r_w}}|a_{x,y}-\mathbb{E}[a_{x,y}]| \le Cr_w^{1/2}\log^{3/2} r_w.
\end{align}
\end{proposition*}

To prove this proposition, we first need the following lemma. 
Let $c_1,c_2,c_2'$ be as defined in Section~\ref{Subsec:mainresults} and $c \in [c_1,c_2]$ be a constant.
Let $A_1$ be the event that for every pair of nodes in  $\mathcal{N}_{r_w}$, the first passage time is at most $2^{c'_2\size}$. To see that this event occurs w.h.p., let $x,y \in \mathcal{N}_{r_w}$, we first note that  
\begin{align*}
P(a_{x,y} > 2^{c'_2\size} ) \le P(a''_{x,y} > 2^{c'_2\size}),
\end{align*}
where $a''$ is the first passage time of the second intermediate FPP. Now it is easy to see that standard concentration bounds imply that there exists a constant $C'$ such that the last term in the above inequality is at most $\exp(-C'2^{c'_2\size})$. It follows that event $A_1$ occurs w.h.p.

\begin{definition}[Optimal Path] 
A path from node $x$ to node $y$ is called \textit{optimal} if its passage time is equal to the first passage time between these nodes.
\end{definition}

\begin{lemma*}\label{Lemma:lemma1'} 
Let $\mathcal{N}_{r_w}$ be a neighborhood with radius $r_w$ centered at the origin.
If (\ref{eq:1.1}) and (\ref{eq:1.14}) hold, then conditional on event $A_1$ as defined above, there exist constants $0 < C_1,C_2 < \infty$ such that for all $x,y\in \mathcal{N}_{r_w}$
\begin{align}\label{eq:2.29'}
P\{ \pi_{x,y} \not \subseteq \mathcal{N}'_{r_w(\log r_w)^{C_1}} \} \le e^{-C_2{r_w}},
\end{align}
where $\pi_{x,y}$ denotes the optimal path between $x,y$ and $\mathcal{N}'_{r_w(\log r_w)^{C_1}}$ denotes a neighborhood with radius $r_w(\log r_w)^{C_1}$ centered at the origin.
\end{lemma*}
The proof of the above lemma is similar to that of Lemma~\ref{Lemma:lemma1}
 and is omitted. 
\begin{proof}[Proof of Proposition~\ref{Prop:Tess_fluc}]
 Let $A_2$ denote the event that the optimal path connecting any two   nodes in $\mathcal{N}_{r_w}$ is contained in $\mathcal{N}'_{r_w(\log r_w)^{C_1}}$. Using Lemma~\ref{Lemma:lemma1'} this event occurs w.h.p.

Let $C_3$ be a constant to be determined later. Let $w$ be sufficiently large so that $P(A_1)$ is greater than $1/2$. Let $w$ be large enough such that for $\|x-y\| \ge 2^{(c/4)\size}$, (\ref{eq:1.15}) holds. Using (\ref{eq:1.15}) for these nodes and standard concentration bounds for nodes such that $\|x-y\| < 2^{(c/4)\size}$ it is easy to see that for all $x,y$ in $\mathcal{N}_{r_w}$, there exists a constant $C_4$, such that for sufficiently large $\size$,
\begin{align*}
P\left(\left| a_{x,y}- \mathbb{E}[a_{x,y}]  \right|^2 \ge C_3r_w\log^{3} r_w \right) \le 2 \exp(-c'C_4\log r_w).
\end{align*}
 where $c'$ is the constant in (\ref{eq:1.15}). Now letting $C_3=6/c'$, we have that for large enough $w$, all $x,y$ such that $\|x-y\|<r_w$, 
\begin{align*}
P\left(\left|  a_{x,y}- \mathbb{E}[a_{x,y}]  \right|^2 \ge C_3r_w\log^{3} r_w \right) \le 2r_w^{-6}. 
\end{align*}
Hence for $w$ large enough we have
\begin{align*}
P\left( \sup_{x,y\in \mathcal{N}_{r_w}}\left|  a_{x,y}- \mathbb{E}[a_{x,y}]  \right|^2 \ge C_3r_w\log^{3} r_w \right) \le 2r_w^{-6}|\mathcal{N}_{r_w}|^2 \\ \le 16r_w^{-2}\le 16w^{-2}.
\end{align*}
Using the second Borel-Cantelli lemma, the result follows from the fact that event $V$ defined in Section~\ref{Sec:Concentration} occurs w.h.p. and $\sum_{w=1}^{\infty} w^{-2} < \infty$.
\end{proof}

Now let $\ti \in [2^{c_1\size},2^{c_2\size}] $. It follows from Proposition~\ref{Prop:Tess_fluc} that almost surely there exists $C'$, such that for sufficiently large $\size$, we have
\begin{align}\label{eq:Tess_avg}
\bar{A}(0,\ti-C'\ti^{1/2}\log^{3/2} \ti) \subset A'_F(0,\ti) \subset \bar{A}(0,\ti+C'\ti^{1/2}\log^{3/2} \ti).
\end{align}
We now  focus on the second part of the proof of Proposition~\ref{Prop:Tess1}.

\begin{theorem*}\label{Thrm:Tess_avg_distance} Let $\ti \in [2^{c_1\size},2^{c_2\size}] $. 
Almost surely, there exist $c>0$ and $\size_0$ such that, for every $\size > \size_0$, there exists a norm $l_*$ such that, 
\begin{align}\label{eq:Tess2}
B_{l_*}(0, \ti-\size^c \ti^{1/2}\log^{3/2} \ti)\cap \mathbb{Z}^2 \subset \bar{A}(0,\ti) \subset B_{l_*}(0,\ti+\size^c \ti^{1/2}\log^{3/2} \ti).
\end{align} 
\end{theorem*}

The proof of the above theorem follows from the following results.
We first need the following lemma.

\begin{lemma*}\label{Lemma:pre_expandable}
Consider an arbitrary $10w$-neighborhood denoted by $\mathcal{N}_{10w}$ in the initial configuration of the modified model. Assume that we place an affected* block outside of $\mathcal{N}_{10w}$, and at some time $\ti^*>0$ there is a $\theta$-affected node $u$ at the center of $\mathcal{N}_{10w}$. Then, at time $\ti^*$, w.h.p. there exists a sequence of flips of $\theta$-particles in $\mathcal{N}_{10w}$ such that if they happen there will be a $\theta$-affected $w$-block inside $\mathcal{N}_{10w}$.
\end{lemma*}
 
To prove the above lemma we first need the following definition.

\begin{definition}[Balance Property]
Let $\mathcal{R}_{a,b}(v) \subset \mathbb{R}^2$ denote a rectangle of sides $a$ and $b$ centered at $v\in G_w$. Let $\mathcal{I}$ be the collection of sets of particles in the possible intersections of the two rectangles $\mathcal{R}_{0.5,w^{1/4}}$  and  $\mathcal{R}_{w^{1/4},0.5}$  with the neighborhood $\mathcal{N}_{10w}$ defined above in the initial configuration. Namely,
\begin{align*}
\mathcal{I} := \bigcup_{v\in G_w} \left\{\mathcal{R}_{0.5,w^{1/4}}(v) \cap \mathcal{N}_{10w} \right\} \cup \bigcup_{v\in G_w} \left\{\mathcal{R}_{w^{1/4},0.5}(v) \cap \mathcal{N}_{10w} \right\}.
\end{align*} 
Also, let $W_I$ be the random variable representing the number of particles in state $\bar{\theta}$ for all $I \in \mathcal{I}$, and ${\size}_I$ be the total number of particles in $I \in \mathcal{I}$. We say a neighborhood has the \textit{balance property} if and only if,
for any $\epsilon \in (0,1/2)$, and for all $I\in \mathcal{I}$ we have $W_I-{N}_I/2  < {w}^{1/8+\epsilon}$.
\end{definition}

\begin{proof}[Proof of Lemma~\ref{Lemma:pre_expandable}] 
Let $A$ be the event that the balance property holds for $\mathcal{N}_{10w}$ in the initial configuration. With a similar argument as in Lemma~\ref{Lemma:goodblock}, and since adding conditions on the configuration of $\mathcal{N}_{10w}$ in the modified model can only increase the probability of having  balanced neighborhoods, it is easy to see that $A$ occurs w.h.p. 

 Now assume that the balance property holds for $\mathcal{N}_{10w}$ in the initial configuration. Consider a $w$-block (henceforth referred to as a block) inside $\mathcal{N}_{10w}$ such that the affected node at the center of $\mathcal{N}_{10w}$ denoted by $u$ at time $\ti^*$ is located in the lower-left corner of this block. We then partition $G_w$ into blocks starting from this block (see Figure~\ref{fig:goodbad}). 
 Since node $u$ is an affected node, the balance property in the initial configuration implies that there have been $$N' = (1/2-\tau)\size+o(\size)$$ flips in the neighborhood of $u$. The number of particles in this neighborhood is close to the number of particles in four blocks, since the number of particles in a block differs from the number of particles in  a quarter of a neighborhood only by $O(w)$.
 This in turn implies that there have been $\size'$ p-stable particles in these blocks that have flipped, so that node $u$ has become affected. Let $B$ denote the event that the sequence of flips mentioned in the statement of this lemma exists.
 
Now consider the first p-stable particle in one of these four blocks located at some node $u_1$. In order for this particle to become p-stable, there should have been $\size'$ flips of p-stable particles in the rest of its neighborhood. The rest of the neighborhood  consists of at most three quarter of a neighborhood, so the number of particles in it is close to the number of particles in three blocks, since the number of particles in a block differs from the number of particles in a quarter of a neighborhood only by   $O(w)$.
Next, consider the first particle that has become p-stable in one of these three blocks (located at some node $u_2$). With a similar argument as above, there should have been $\size'$ flips of p-stable particles in the rest of the neighborhood of this particle, which again consists of at most three blocks to make this particle p-stable. 
We now want to consider an event whose occurrence will lead to $u_1$ requiring the maximum number of flips to become affected. Let $f_1, f_2,$ and $f_3$ denote the number of flips in each of the three blocks that happen before $u_1$ becomes affected, respectively. Let $B_1$ denote the event where $ f_1 = f_2 \pm O(w)$, $f_2 = f_3 \pm O(w)$ and that the $l_\infty$ distance between each pair of the nodes with one node in one of these three blocks and the other in another one of these three blocks is maximum, and that also the $l_\infty$ distance between each of these nodes where a flip occurs and the node $u_1$ is also maximum.  

 Now, if we show that event $B$ conditional on $B_1$ occurs w.h.p., it is easy to see (using the Bayes' theorem and the FKG inequality) that event $B$ occurs w.h.p. as well, namely 
 \begin{align}\nonumber
     P(B) = P(B_1)P\left(B\given[\Big]B_1\right) + P(B_1^C)P\left(B\given[\Big]B_1^C\right) \ge P(B|B_1).
 \end{align}
 
 To see that $B$ conditional on $B_1$ occurs w.h.p., we first argue that the block with the largest number of p-stable particles in at most the three blocks comprising close to $3/4$ of the neighborhood of node $u_1$ -- except for $O(w)$ nodes -- where flips could have happened has had at least $$(1/3)\size' + o(\size')$$  p-stable particles. 
 Let $u_2$ be the first node in the neighborhood of $u_1$ that has become p-stable. Let $B_2$ denote a similar event as $B_1$ but for node $u_2$. 
 With a similar argument for the above equation, it is easy to see that event $B$ has a smaller probability when event $B_2$ occurs as well.
 \begin{align}\nonumber
     P(B) \ge P\left(B \given[\Big] B_1, B_2 \right)
 \end{align}
 
 It is now easy to check (using similar elementary arguments as in the proof of Lemma~5 in \cite{omidvar2017self2}) that the flips of particles leading to node $u_2$ becoming p-stable and the flips of particles leading to node $u_1$ becoming p-stable will lead to the formation of a block of affected nodes centered at this node. 
Finally, combined with the fact that the balance property holds for $\mathcal{N}_{10w}$ w.h.p., the proof is complete.
\end{proof}

\begin{proposition*} \label{Prop:Tess_avg_SAG} 
There exists a constant $c>0$ such that for sufficiently large $\size$, $\mathbb{E}[a]$ is SAG($Q$) where
\begin{align} 
Q(\alpha) = \frac{\alpha^{1/2}}{\size^c\log^{3/2} \alpha}.   
\end{align}
\end{proposition*}   
\begin{proof}


Let $ r_w = 2^{c'_1\size}$ where $c'_1$ is the constant in the definition of SAG. Let $\mathcal{N}_{v}$ be as defined in Section~\ref{Sec:Concentration}. Let events $A_1, A_2$ be as defined in the proof of Proposition~\ref{Prop:Tess_fluc}. Let $A_3$ be the event that the first flipping time of all the particles inside $\mathcal{N}_{v}$ is  less than $ {\size^2}$. Let $t_x$ denote the first flipping time of particle at node $x \in \mathcal{N}_v$. Now \ref{eq:1.14} implies that for sufficiently large $\size$ there exists a constant $c>0$ such that
$$ P(t_x> \size^2 ) \le \exp(-c {\size^2}). $$
It follows that event $A_3$ occurs w.h.p.  

Let us define a \textit{very good block} as a $10w$-block that satisfies the balance property and a \textit{very bad block} as a $10w$-block that does not satisfy this property. Let us divide the lattice into $10w$-blocks starting from the block centered at the origin. Since each of these blocks is a very good block w.h.p., it follows from Theorem~\ref{Thrm:grimmett_bad_cluster} that in $\mathcal{N}_v$ there are no clusters of very bad blocks with radius larger than $\size^2$ in the initial configuration w.h.p.  Let us denote this event by $A_4$. Conditional on this event, using Lemma~\ref{Lemma:pre_expandable}, we can conclude that when a node becomes affected after a time $O(\size^3)$ there is an affected block in a neighborhood of radius $\size^3$ w.h.p.

Now we can conclude that the intersection of the above events occurs w.h.p. In particular, w.h.p. we have
\begin{align}\label{eq:Tess_sup}
\sup_{x,y\in \mathcal{N}_{r_w}}\left|  a_{x,y}- \mathbb{E}[a_{x,y}]  \right| \le Cr_w^{1/2}\log^{3/2} r_w,
\end{align}
the optimal path between any two nodes inside $\mathcal{N}_{r_w}$ is contained in $\mathcal{N}_v$, and the first flipping time of all the particles inside $\mathcal{N}_v$ is less than $\size^2$.

%
%

 Consider an optimal path $\gamma$ between $x,y$  where $ x,y \in \mathcal{N}_{r_w}$.  
 It follows that the maximum flipping time over all the particles on $\gamma$ is at most  $\size^2$ w.h.p. Therefore w.h.p. one can find a node $z$ in $\gamma$ and $\lambda \in (0,1)$ such that we have
 \begin{align*}
 \left|\lambda a_{x,y} -a_{x,z}\right| \le \size^2,
 \end{align*}
 or equivalently
 \begin{align*}
 \left|(1-\lambda) a_{x,y} -(a_{x,y}-a_{x,z})\right| \le \size^2.
 \end{align*}
 
Now let us denote $a_{x,y}-a_{x,z}$ by $a^*_{z,y}$. We now show that w.h.p.
\begin{align}\label{eq:sub_ad}
a^*_{z,y} \le  a_{z,y} + \size^5.
\end{align}
To see this, we note that the event of $a_{z,y}$ being smaller than some value and the event of having an affected block centered at $z$ are both increasing in the change of a $\theta$-particle to a $\bar{\theta}$-particle, so using the FKG inequality these events are positively correlated. Now we note that since $z$ has become affected after $a_{x,z}$, by considering an upper bound for the fliping times that consists of the summation of all the flipping times and using  standard concentration bounds we can conclude that after at most $\size^5$ time, w.h.p. this node will be inside an affected block and hence we can conclude that (\ref{eq:sub_ad}) holds w.h.p.


Next, we show that there exists a constant $c>0$ such that $a^*_{z,y} \ge a_{z,y} - \size^cr_w^{1/2}\log^{3/2} r_w$ w.h.p. To see this, first we note that due to (\ref{eq:Tess_sup}) and (\ref{eq:bilips}) we can conclude that there exists $c'>0$ such that there are no affected nodes of any type in $\mathcal{N}_v\setminus [A(x,a_{x,y})]_{\size^{c'}r_w^{1/2}\log^{3/2} r_w}$, where $[A]_t$ denotes the $l_\infty$ neighborhood of $A$. 
Consider a neighborhood $\mathcal{N}(z)$ with radius $\size^{c'}r_w^{1/2}\log^{3/2} r_w$ centered at $z$. We want to argue that w.h.p. the formation of the optimal path from $x$ to $y$ from the time $a_{x,z}$ only involves the spread of affected nodes from $\mathcal{N}(z)$. To see this, assume that the spread of affected nodes started from a node $z'$ contained in $A'(0,a_{x,z})\setminus \mathcal{N}(z)$ has also participated in the formation of the optimal path. This implies that w.h.p. there had been a sequence of possible flips leading to a p-stable particle that was needed for the formation of the optimal path before the spread of affected nodes from $\mathcal{N}(z)$ had reached this particle. However, this implies that w.h.p. node $z$ is not on the optimal path from $x$ to $y$ which is a contradiction. Hence we can conclude that there exists $c>0$ such that  $a^*_{z,y} \ge a_{z,y} - \size^cr_w^{1/2}\log^{3/2} r_w$ w.h.p.

Now we can conclude that there exists a constant $C''>0$  such that w.h.p.,
 \begin{align*}
 \left|\lambda a_{x,y} -a_{x,z}\right| \le \size^2,
 \end{align*}
 and
 \begin{align*}
 \left|(1-\lambda) a_{x,y} -a_{z,y}\right| \le \size^{C''}r_w^{1/2}\log^{3/2}r_w.
 \end{align*}

Combined with (\ref{eq:Tess_sup}) we can conclude that, w.h.p. there exists a constant $D>0$ such that for all $x,y \in \mathcal{N}_\rho$ there exists a node $z$ for which we have
\begin{align*}
\left|\lambda\mathbb{E}[a_{x,y}]-\mathbb{E}[a_{x,z}]\right| \le \size^Dr_w^{1/2}\log^{3/2}r_w,
\end{align*}
and
\begin{align*}
\left|(1-\lambda)\mathbb{E}[a_{x,y}]-\mathbb{E}[a_{z,y}]\right| \le \size^Dr_w^{1/2}\log^{3/2}r_w.
\end{align*}
The result now follows from Propositions 3.1 and 4.1 in Sections~3 and~4 in \cite{tessera2014speed}.
\end{proof}

Now, in order to use the above definition to conclude the shape theorem we need the following result which is a modified version of Proposition~1.8 in \cite{tessera2014speed}. 


\begin{proposition*}[Tessera -- Modified]\label{Prop:SAG_to_shape}
 If $\mathbb{E}[a]$ is $SAG(Q)$ with
\begin{align*}
Q(\alpha) = \frac{\alpha^{1/2}}{\size^c\log^{3/2} \alpha},
\end{align*} 
 then there exists a norm $l_*$ on $\mathbb{R}^2$ and $C>0$ such that for all $\ti \ge 2^{c_1\size}$, when $\size$ is sufficiently large we have
\begin{align*}
B_{l_*}\left(0,\ti-\size^C \ti^{1/2}\log^{3/2} \ti \right)  \cap \mathbb{Z}^2 \subset \bar{A}(0,\ti) \subset B_{l_*}\left(0,\ti+\size^C\ti^{1/2}\log^{3/2}\ti \right) .
\end{align*}
\end{proposition*}

\begin{proof}
The proof is similar to the proof of Proposition~1.8 in  \cite{tessera2014speed}, except for the fact that we need to prove a statement similar to Lemma~5.1 in \cite{tessera2014speed} for our setting (since we do not have triangular inequality for $\mathbb{E}[a]$). We want to show that there exists a constant $c''>0$ such that for all $W\in \mathbb{N}$ and all $r/W \ge 2^{c'_1\size}$ where $c'_1$ is the constant in the definition of SAG, and for $\size$ sufficiently large we have,
\begin{align*}
d_H\left( \frac{1}{r} \bar{A}(0,r/W)^W,\frac{1}{r}\bar{A}(0,r) \right) \le \frac{\size^{c''}}{Q(r/W)},
\end{align*}
where $d_H$ denotes the  Hausdorff distance with respect to the $l_2$ norm. 
We note that due to (\ref{eq:Tess_EQSAG1}) we can write 
\begin{align*}
 \bar{A}(0,(1-\epsilon)r/W)^W \subset \bar{A}(0,r) \subset \bar{A}(0,(1+\epsilon)r/W)^W, 
\end{align*}
where
\begin{align*}
\epsilon = \frac{1}{Q(r/W)}
\end{align*}
Since there exists a $c''>0$ such that $\|x-y\| \le \size^{c''}\mathbb{E}[a_{x,y}]$ for all $x,y\in \mathcal{N}$ we have,
 \begin{align*}
\bar{A}(0,(1+\epsilon)r/W)^W \subset (\bar{A}(0,(1-\epsilon)r/W)_{12\epsilon r/W})^W \\
\subset  (\bar{A}(0,(1-\epsilon)r/W))B_{l_2}(0,12\size^{c''}\epsilon r/W)^W \\
= (\bar{A}(0,(1-\epsilon)r/W))^WB_{l_2}(0,12\size^{c''}\epsilon r)
 \end{align*}
where the notation $[\bar{A}(0,r/W)]_{12\epsilon r /W}$ stands for the $12\epsilon r/W$-neighborhood of $\bar{A}(0,r/W)$. The rest of the proof follows from the first part of the proof of Proposition~1.8 in~\cite{tessera2014speed} presented in Section~5.3.
\end{proof}

\begin{proof}[Proof of Proposition~\ref{Prop:Tess1}]
The proof follows combining Propositions~\ref{Prop:Tess_fluc}, \ref{Prop:Tess_avg_SAG}, and~\ref{Prop:SAG_to_shape}.
Using Proposition~\ref{Prop:Tess_fluc}, we can conclude that, almost surely, there exist $N'_0$ and $C'>0$ such that for $\size > N'_0$, we have
\begin{align*}
\bar{A}(0,\ti-C'\ti^{1/2}\log^{3/2} \ti) \subset A'_F(0,\ti) \subset \bar{A}(0,\ti+C'\ti^{1/2}\log^{3/2} \ti).
\end{align*}

Using Proposition~\ref{Prop:Tess_avg_SAG} we can conclude that $\mathbb{E}[a]$ is SAG($\alpha^{1/2}/\size^c\log^{3/2} \alpha$). Finally, the proof follows from Proposition~\ref{Prop:SAG_to_shape} with ${Q(\ti)=\ti^{1/2}/\size^c{\log^{3/2} \ti}}$ when $\size$ is sufficiently large.

\end{proof}

\begin{lemma*} \label{Lemma:Unhappy_growth}
Let $c_2<c'_2<c_2'' \in (c_2,0.5(1-H(\tau')))$ be constants and let $u\in G_w$, and $\mathcal{N}_{\radius}(u)$ and $\mathcal{N}_{\radius'}(u)$ be two neighborhoods with radii ${\radius}=2^{c'_2\size}$ and $\radius'=2^{c''_2\size}$. 
There exists a constant $c \in \mathbb{R}^+$, such that for sufficiently large $\size$, we have
\begin{align*}
P\left(\mathcal{T}(\radius)>\rho\right) > 1 - e^{- \sqrt{\radius'}/\size^c},
\end{align*}
where $\mathcal{T}(\radius)$ is defined in (\ref{Tinfdef}). 
\end{lemma*} 
\begin{proof}
Let $x,y$ be two nodes on $\mathcal{N}_\rho(0)$ and $\mathcal{N}_{\rho'}(0)$ respectively such that they have the largest $l_\infty$ distance from the origin (i.e., they are on the boundaries of these neighborhoods).  
Let $a'$ denote the first passage time between two nodes located on $\mathbb{Z}^2$, similar to (\ref{eq:bilips}) there exists a constant $c'>0$ such that for sufficiently large $\size$ we can write
\begin{align*}
P\left(a_{x,y} \le \rho\right) \le P\left(a'_{0,\xi_1(\rho-\rho')/\size^{c'}} \le \rho \right).
\end{align*}
Now due to (\ref{eq:bilips}) there exists a constant $c''>0$ such that we can write
\begin{align*}
P(a_{x,y} \le \rho) \le P\left(a'_{0,\xi_1(\rho-\rho')/\size^{c'}}-\mathbb{E}\left[a'_{0,\xi_1(\rho-\rho')/\size^{c'}}\right] \le \rho - \frac{\rho'-\rho}{\size^{c''}}\right).
\end{align*}
Now using the fact that there are less than $64\rho'^2$ pairs of nodes on the boundaries of $\mathcal{N}_{\rho}(u)$ and $\mathcal{N}_{\rho'}(u)$ we can use the Talagrand concentration bound \cite[Proposition~8.3]{talagrand1995concentration} to conclude that there exists a constant $c>0$ such that for sufficiently large $\size$ we have
\begin{align*}
P(\mathcal{T}(\rho) \le \rho) \le \exp(-\sqrt{\rho'}/\size^c).
\end{align*}
\end{proof} 

We are now ready to consider the original model   introduced in this paper and provide the proof of Theorem~\ref{Thrm:Shape_theorem}.

\begin{proof}[Proof of Theorem 1 (Shape Theorem)] 
Let $\mathcal{N}_{\size}$ and $\mathcal{N}_{2\size}$ be two neighborhoods with radii $\size$ and $2\size$ respectively that are centered at the origin. Let us form a grid of $w$-blocks such that we have a block centered at the origin.  Let $A$ be the event of having a $\theta$-affected block at the origin in the original model. Let $A_1$ be the event that all the blocks in $\mathcal{N}_{2\size}\setminus \mathcal{N}_{\size}$ are good.  Using Lemma~\ref{Lemma:goodblock}, $A_1$ occurs w.h.p. Let $c'_2, c''_2 \in(c_2,0.5(1-H(\tau')))$ be  constants such that $c''_2>c'_2$ and let $\mathcal{N}_{\rho'}(0)$ and $\mathcal{N}_{\rho''}(0)$ be neighborhoods with radii $\rho' = 2^{c'\size}$ and $\rho'' = 2^{c''\size}$ respectively.
Let $A_2$ be the event that there are no affected nodes of any type in $\mathcal{N}_{\rho''} \setminus \mathcal{N}_{2\size}$. By Lemma~\ref{Lemma:R_unhappy}, $A_2$ occurs w.h.p. 
Let $A_3$ be the event that $\mathcal{N}_{\rho''}$ is a region of expansion.
Using Lemma~\ref{Lemma:monoch_spread_1}, the fact that events $A$ and $A_3$ are increasing events in the change of a $\theta$-particle to a $\bar{\theta}$-particle, and the FKG inequality, we can conclude that event $A_3$ occurs w.h.p.
Let $A_4$ be the event that, conditional on the previous events, there are no affected nodes of any type inside $\mathcal{N}_{\rho'}$ that is due to the spread of any type of affected node outside of $\mathcal{N}_{\rho''}$ at any time $\ti \le \rho'$.
Using Lemma~\ref{Lemma:Unhappy_growth}, and an application of FKG inequality, this event occurs w.h.p. 


Additionally, 
let $A_5$ and $A'_5$ be the events that the sum of the first flipping times of the particles in $\mathcal{N}_{2\size}$ in the original process and in the modified process will be less than $\sqrt{\ti}$ respectively where $t$ is as defined in the statement of the theorem. Standard bounds imply that $A_5$ and $A'_5$ occur w.h.p. 
Now, conditional on $A, A_1, A_2, A_3,A_4$ we can consider the following coupling between the original model and the modified model. We assume that all the particles in the initial configuration that are located outside $\mathcal{N}_{2\size}$ in both models have the same states, and that the first flipping time of each particle in the original model is the same as the flipping time of its corresponding particle in the modified model.
Using this coupling we can conclude that there exists $c>0$ such that 
\begin{align*}
&P\left( A_{F,\rho}(0,\ti) \subset B_{l_*}(0,\ti+\size^{c}\ti^{1/2}\log^{3/2} \ti) \given[\Big] A,A_1,A_2,A_3,A_4   \right) \\
&\ge P\left( A'_{F,\rho}(0,\ti + \sqrt{\ti}) \subset B_{l_*}(0,\ti+\size^{c}\ti^{1/2}\log^{3/2} \ti) \given[\Big] A'_5 \right),
\end{align*}
and
\begin{align*}
&P\left(B_{l_*}(0,\ti-\size^{c}\ti^{1/2}\log^{3/2} \ti)\cap \mathbb{Z}^2 \subset  A_{F,\rho}(0,\ti) \given[\Big] A,A_1,A_2,A_3,A_4,A_5   \right) \\
&\ge P\left(B_{l_*}(0,\ti-\size^{c}\ti^{1/2}\log^{3/2} \ti)\cap \mathbb{Z}^2 \subset  A'_F(0,\ti-\sqrt{\ti}) \right).
\end{align*}
Since events $A_1,A_2,A_3,A_4,A_5,A_5'$ occur w.h.p., and according to Proposition~\ref{Prop:Tess1}, event $B_{l_*}(0,\ti-\size^{c}\ti^{1/2}\log^{3/2} \ti)\cap \mathbb{Z}^2 \subset A'_{F,\rho}(0,\ti \pm \sqrt{\ti}) \subset B_{l_*}(0,\ti+\size^{c}\ti^{1/2}\log^{3/2} \ti)$ occurs w.h.p., hence we can conclude that, in the original model, conditional on $A$, w.h.p. we have
\begin{align*}
 B_{l_*}(0,\ti-\size^{c}\ti^{1/2}\log^{3/2} \ti)\cap \mathbb{Z}^2 \subset A_{F,\rho}(0,\ti) \subset B_{l_*}(0,\ti+\size^{c}\ti^{1/2}\log^{3/2} \ti).
\end{align*} 
%
\end{proof}

%

\section{Proof of the Size Theorem}\label{Sec:Proof}

Without loss of generality, we assume that the set of nodes of $G_w$ is a subset of the nodes on $\mathbb{Z}^2$ and we work with the obvious probability space. 
We first define the \textit{expandable region}; to do so we need the following lemma.


\begin{figure}[!t] 
\centering
\includegraphics[width=2.4in]{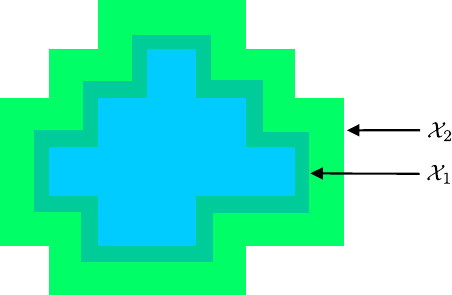}
\caption{Expandable region.  $\mathcal{X}_1$ is called a $\theta$-\textit{expandable region}  if there exists a set of flips of $\theta$-particles inside $\mathcal{X}_1$ leading to a $\theta$-affected node in $\mathcal{X}_2$.}
\label{fig:expandable}
\end{figure}

\begin{lemma*}\label{Lemma:bad_cluster}
For any $c>0$, there are no clusters of bad blocks with radius greater than $\csize$ in a neighborhood with radius $\rho = O(2^{cN})$ in the initial configuration w.h.p.
\end{lemma*}
\begin{proof} 
Let $p$, 
be the probability of having a bad block, and let $k=\csize$. By Theorem~\ref{Thrm:grimmett_bad_cluster}, it follows that w.h.p. there is no cluster of bad blocks containing a bad block with $l_1$-distance from its center greater than $\csize$ in a neighborhood with exponential radius in ${\size}$. 
\end{proof} 

Now divide $G_w$ into $\msize$-blocks and consider the union of particles inside a cluster of bad $\msize$-blocks and the set of particles outside the cluster whose $l_\infty$ distance to at least one node in the cluster is less than or equal to $\msize/4$.  Note that for sufficiently large $\size$, the probability of having a bad $\msize$-block is below the critical probability of percolation, and each $\msize$-block is a bad $\msize$-block independently of the others, hence by Lemma~\ref{Lemma:bad_cluster}, w.h.p. there is no cluster of bad $\size$-blocks with radius larger than $\csize$ in a neighborhood with exponential size in $\size$ on $G_w$. Let $\mathcal{X}_1$ denote an arbitrary set of bad $\size$-blocks and its outer boundary as defined above such that its radius is smaller than $\size^3$. Also, consider the set of all the particles outside $\mathcal{X}_1$ whose $l_\infty$ distance to at least one particle in $\mathcal{X}_1$ is less than or equal to $\msize/4$ and denote it by $\mathcal{X}_2$.

\  

\begin{definition}[Expandable Region]   $\mathcal{X}_1$ is called a $\theta$-\textit{expandable region}  if  there exists a set of flips of $\theta$-particles inside $\mathcal{X}_1$ leading to a $\theta$-affected node in $\mathcal{X}_2$ (see Figure~\ref{fig:expandable}). It is noted that if $\mathcal{X}_1$ is not an expandable region, the possible spread of $\theta$-affected nodes started in it will die out before reaching $\mathcal{X}_2$. The center of an expandable region is the node at the center of the smallest neighborhood that contains the expandable region.

\end{definition}

\begin{remark} It is noted that there is a difference between an expandable region and an expandable radical region described in Definition~\ref{Def:exradreg}. While an expandable radical region is a radical region that can make a $\mathcal{N}_{w/2}$-neighborhood monochromatic, an expandable region is capable of starting the spread of affected nodes beyond a ``local'' region as mentioned in the following lemma.
\end{remark}

We now want to argue about the distance of the closest expandable region to the origin. The following lemma shows how far the closest expandable region to the origin can be. We show this by establishing a relationship between radical regions and expandable regions. The following lemma exploits the fact that the closest radical region to the origin is also an expandable region w.h.p.

\begin{lemma*}\label{Lemma:rho}
Let $\epsilon > 0$ and $\epsilon' > e(\tau)$ as defined in (\ref{eq:ftau}). W.h.p. the $l_*$-distance of the origin from the node  at the center of the closest expandable region in the initial configuration that can make the origin affected is at most 
\begin{align*}
\rho =2^{0.5(1-H(\tau)+\epsilon)(1+\epsilon')^2N}.
\end{align*}
\end{lemma*}

\begin{proof}
Consider a neighborhood with radius $2^{0.5(1-H(\tau)+\epsilon/2)(1+\epsilon')^2N}$ centered at the origin. Using Lemma~\ref{Lemma:radical_region_prob}, w.h.p. there exists a radical region in this neighborhood. 
Using Lemma~\ref{Lemma:Trigger} a radical region is expandable w.h.p. 
Using Lemma~\ref{Lemma:bad_cluster}, w.h.p. there is no cluster of bad blocks with radius larger than $\csize$ in a neighborhood with radius $\rho$ centered at the origin in the initial configuration. This implies that the expandable radical region is not surrounded by bad blocks that can potentially stop its spread.
Now consider a neighborhood with radius $3\csize$ centered at the center of the radical region. Since the event of having an expandable radical region at the center of this neighborhood and the event of this region being a region of expansion are both increasing events in the change of a $\theta$-particle to a $\bar{\theta}$-particle, by an application of  FKG inequality \cite{fortuin1971correlation} and using Lemma~\ref{Lemma:monoch_spread_1} this neighborhood is a region of expansion w.h.p. Now, since the expandable radical region can turn the entire region of expansion monochromatic, this implies that this region is an expandable region and this completes the proof.
\end{proof}

The following lemma shows that in the initial configuration w.h.p. there is no part of an expandable region in an annulus around the origin whose width is $\rho'$.

\begin{lemma*}\label{Lemma:rhop}
Let $\epsilon > 0$ and $\epsilon' > e(\tau)$ as defined in (\ref{eq:ftau}). W.h.p. there is no node that belongs to an expandable region in \begin{align*}
B_{l_*}{(0,\rho+\rho')}\setminus B_{l_*}{(0,\rho)},
\end{align*}
 for
\begin{align*}
\rho =2^{0.5(1-H(\tau)+\epsilon/2)(1+\epsilon')^2\size}, \\
\rho' = 2^{(1-H(\tau)-\epsilon)(1-0.5(1+\epsilon')^2)\size + o(\size)}.
\end{align*}  
\end{lemma*}

\begin{proof}
 By Lemma~\ref{Lemma:bad_cluster}, w.h.p. there is no cluster of bad blocks with radius larger than $\csize$ in a neighborhood with radius $2(\rho+\rho')$ centered at the origin in the initial configuration. Also, for large $\size$ we have,
\begin{align*}
\mbox{Number of nodes in } B_{l_*}{(0,\rho+\rho'+2\csize)}\setminus B_{l_*}{(0,\rho)} \le  8\rho \rho'.
\end{align*}
Also,  if we have 
\begin{align*}
\mbox{Number of nodes in } B_{l_*}{(0,\rho+\rho'+2\csize)}\setminus B_{l_*}{(0,\rho)} = 2^{(1-H(\tau)-\epsilon)N+o(N)},
\end{align*}
then w.h.p. there will be no location inside $B_{l_*}{(0,\rho+\rho'+2\csize)}\setminus B_{l_*}{(0,\rho)}$ for which a particle would be p-stable. Here we have used the fact that $(H(\tau) - H(\tau'')) = o(1)$, and the fact that having a p-stable particle is a necessary condition for having an expandable region.
\end{proof}


The following lemma shows that w.h.p. an expandable region can lead to the formation of a $\theta$-affected $w$-block.

\begin{lemma*}\label{Lemma:expandable_to_m'}
W.h.p. there exists a sequence of possible flips in $\mathcal{X}_1 \cup \mathcal{X}_2$ that can lead to a $\theta$-affected $w$-block centered in $\mathcal{X}_1$. 
\end{lemma*}

\begin{proof}
Since $\mathcal{X}_1$ is an expandable region, there exists a sequence of flips leading to a new affected node inside a good $\msize/4$-block. By coupling with our modified model and using Lemma~\ref{Lemma:pre_expandable} we can conclude that, w.h.p. there exists a sequence of possible flips in the $\msize/4$-block that can lead to a monochromatic $w$-block in it. Now with an application of Lemma~\ref{Lemma:monoch_spread_1} and using the fact that the event of having $\mathcal{X}_1 \cup \mathcal{X}_2$ inside a $\csize$-block and the event of having that block   being a region of expansion are positively correlated, we can conclude that the latter event occurs w.h.p. and this completes the proof.  
\end{proof}

We are now ready to begin the proof of Theorem~\ref{Thrm:main_thrm}.

\begin{proof}[Proof of Theorem~\ref{Thrm:main_thrm}]
 We first show that the size of the largest monochromatic ball is at least exponential in $\size$ w.h.p. Let $\epsilon' > e(\tau)$ with $e(\tau)$ as defined in (\ref{eq:ftau}), and $\epsilon'' > 0$, such that
\begin{align*}
&a(\tau) - \epsilon \le  \left(1-H(\tau)-\epsilon''\right) \left(2-(1+\epsilon')^2\right), \\
&b(\tau) + \epsilon \ge \left(1+\epsilon'\right)^2(1-H(\tau)+\epsilon'').
\end{align*}
Let $\ti^* = 2^{(a(\tau) + \epsilon)\size}$. We wish to show that for all $\ti \ge \ti^*$,
\begin{align*}
M_\ti \ge 2^{\left(a(\tau) - \epsilon\right)\size} \mbox{ w.h.p.}
\end{align*}
By Lemma~\ref{Lemma:bad_cluster}, w.h.p. there is no cluster of bad blocks with radius larger than $\csize$ in a neighborhood with radius $2^{\size}$ centered at the origin in the initial configuration (event $A_0$). 

Let 
\begin{align*}
\rho =2^{0.5(1-H(\tau)+\epsilon''/2)(1+\epsilon')^2\size}, \\
\rho' = 2^{(1-H(\tau)-\epsilon'')(1-0.5(1+\epsilon')^2)\size + 2\log_2 \size},\\
\rho'' = 2^{(1-H(\tau)+\epsilon'')\left((1+\epsilon')^2-1\right)\size},
\end{align*} 


We let $\size$ be sufficiently large so that there exists a norm $l_*$ and $C>0$ such that (\ref{eq:Tess1}) in Proposition~\ref{Prop:Tess1} is satisfied for $\ti = \rho'^{1/3}$. We also assume that $\size$ is sufficiently large, such that $\rho'/\size > w^3$, $LL' < \rho'/4 -\rho'^{1/3}-(\size^4+\size)\rho''-\size$ where 
\begin{align*}
L= \left\lceil \frac{\rho}{\rho'/4-2\size^C\sqrt{\rho'}\log^{3/2} \rho' - \size}\right\rceil, \\
L'= \left\lceil 2\size^C\sqrt{\rho'}\log^{3/2} \rho' + \size \right\rceil.
\end{align*}

Let the closest expandable region to the origin in the $l_*$ norm be a {$\theta$-expandable} region and let $\mathcal{N}_{\rho'/4}$ be a neighborhood at the origin with radius $\rho'/4$, now let
\begin{align*}
\mathcal{T}(\rho'/4) = \inf\left\{\ti \  \given[\Big] \  \exists \mbox{ a $\bar{\theta}$-affected node in } \mathcal{N}_{\rho'/4} \right\}, \\
A = \left\{ \mbox{The origin is contained in a firewall of radius $\rho'/\size$ before } \mathcal{T}(\rho'/4)  \right\}.
\end{align*} 

We now want to show
\begin{align} \label{eq:P_A}
A \mbox{ occurs w.h.p.}
\end{align}
To show this, we condition $A$ on a few events that we show to occur w.h.p. and argue that since event $A$ conditional on these events also occurs w.h.p., $A$ occurs w.h.p. 


 Let $X$ denote the $l_*$-distance from the origin to the closest node in an expandable region, that without loss of generality, we have assumed to be of type $\theta$. Let
\begin{align*}
&A_1 = \left\{ \mbox{$X\le \rho$, at $\ti=0$}  \right\},  \\
&A_2 = \left\{ \nexists \mbox{ a $\bar{\theta}$-expandable region in } B_{l_*}(0,X+\rho')\setminus B_{l_*}(0,X) \mbox{ at $\ti=0$} \right\}.
\end{align*}

\begin{figure}[!t] 
\centering
\includegraphics[width=4in]{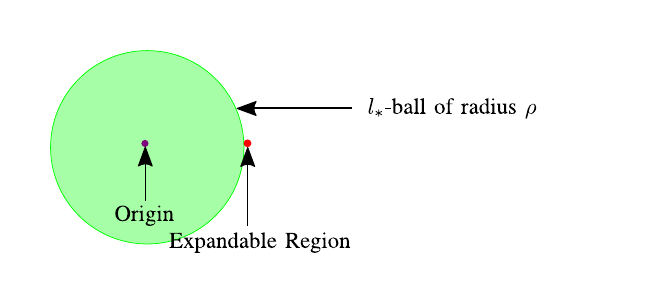}
\caption{W.h.p. the closest expandable region to the origin will not be farther than $\rho$ with respect to the $l_*$-norm. } 
\label{fig:closest_expandable} 
\end{figure}

Consider an $l_*$-ball of radius $\rho$. According to Lemma~\ref{Lemma:rho}, w.h.p. there is an expandable region in this ball (see Figure~\ref{fig:closest_expandable}). This implies
\begin{align*}
A_1 \mbox{ occurs w.h.p.} 
\end{align*}
Using the fact that the existence of a $\theta$-expandable region  in $B_{l_*}(0,X+\rho')\setminus B_{l_*}(0,X)$ can only increase the probability of event $A_2$ (since they are both increasing events in the change of a $\theta$-particle to a $\bar{\theta}$-particle, by an application of FKG inequality \cite{fortuin1971correlation} for the initial configuration they are positively correlated), and the fact that conditional on event $A_1$, event $A_2$ would have the smallest probability when $X=\rho$ (since it implies having an annulus with the largest area), we let
\begin{align*}
&A'_2 = \left\{ \nexists \mbox{ a $\bar{\theta}$-expandable region in } B_{l_*}(0,\rho+\rho')\setminus B_{l_*}(0,\rho) \mbox{ at $t=0$} \right\},
\end{align*} 
and by an application of FKG inequality \cite{fortuin1971correlation} for the initial configuration we have
\begin{align*}
P(A_2) \ge P\left(A_2 \given[\Big] A_1\right) P(A_1) \ge P\left(A_2 \given[\Big] X = \rho\right) P(A_1) \ge P(A'_2)P(A_1).
\end{align*}
Using Lemma~\ref{Lemma:rhop}, event $A'_2$ occurs w.h.p., hence we have
\begin{align*}
A_2 \mbox{ occurs w.h.p.}
\end{align*}

Now consider the line segment from the center of the closest expandable region to the origin. Let  $\mathcal{N}$ denote the set of particles such that their $l_*$-distances from at least one point on the line segment is less than or equal to $2\size^{c'}\rho'$ where $c'$ is the constant $C_1$ in Lemma~\ref{Lemma:lemma1'}. Let
\begin{align*}
\begin{split}
A_3 = \left\{ \nexists \mbox{ a $\bar{\theta}$-affected node in } \mathcal{N} \mbox{ and it is a region} \right. \\
\left. \mbox{ of expansion at $\ti=0$ }  \right\}.
\end{split}
\end{align*}

Since event $A_3$ has the smallest probability when $X=\rho$ (again, since this will lead to having the largest area for $\mathcal{N}$), and since the existence of an expandable region can only increase the probability of this event, using FKG inequality and with an application of Lemma~\ref{Lemma:unstableprob} and Lemma~\ref{Lemma:monoch_spread_1}, we can conclude that 
 \begin{align*}
A_3 \mbox{ occurs w.h.p.}
\end{align*}



 \begin{figure}[!t] 
\centering
\includegraphics[width=4in]{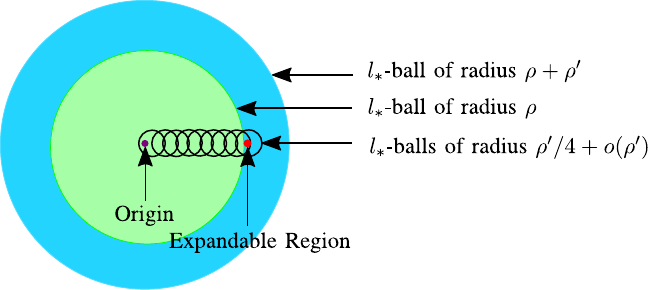}
\caption{The gradual spreads of affected nodes from the expandable region towards the origin in time intervals of $\rho'/4$. After each time interval, all the nodes inside the corresponding $l_*$-ball of radius $\rho'/4-o(\rho')$ are going to be affected.} 
\label{fig:gradual_growth2}
\end{figure} 
We define the \textit{spread} of affected nodes as the supremum of the $l_*$-distances of affected nodes on paths that start from a given node and that have at most a given first passage time.
We are now going to consider a sequence of gradual spreads of the affected nodes first starting from the expandable region and thereafter from the node at the center of an affected block with smallest $l_*$-distance towards the origin in $\rho'/4$ time intervals (see Figure~\ref{fig:gradual_growth}). 

To consider the spread of the $\theta$-affected nodes 
we first notice that using Lemma~\ref{Lemma:expandable_to_m'}, w.h.p. there exists a sequence of less than $\size^4$ flips that can create a $\theta$-affected $w$-block centered at the center of the expandable region. Let us denote this event by $A_4$. Hence we have
\begin{align*}
A_4 \mbox{ occurs w.h.p. }
\end{align*} 
Let $T_{\size}$ denote the time it takes until $\size^4$ flips occur one by one. Let $A'_5 = \{T_{\size} < \rho'^{1/3} \}$.  Standard concentration bounds imply that there exist $c>0$ such that this event occurs with probability at least $1-\exp(-c\rho'^{1/3})$. Let $A_5$ denote the event that the time that it takes until we have a $\theta$-affected block inside the expandable region is less than $\rho'^{1/3}$. We have
\begin{align*}
P(A_5)  \ge P\left(A'_5\right) \ge  1-\exp\left(-c\rho'^{1/3}\right).
\end{align*} 
Hence we have  
\begin{align*}
A_5 \mbox{ occurs w.h.p.}
\end{align*}

Let $A_6$ denote the event of having less than $\rho''$ affected nodes of each type in the $l_*$-ball of radius $\rho$. Standard concentration bounds imply that this event also occurs w.h.p., hence we have 
\begin{align*}
A_6 \mbox{ occurs w.h.p.}
\end{align*}
Now, since we know that w.h.p. there is no $\bar{\theta}$-expandable region inside $B_{l_*}(0,\rho+\rho')$, this implies that the spread   of $\bar{\theta}$-affected nodes is going to die out quickly in this ball and there will not be any spreads of $\bar{\theta}$-affected nodes beyond a radius of $O(\size^3)$ from any of the $\bar{\theta}$-affected nodes. Hence, we will consider the possible spread of $\bar{\theta}$-affected nodes from expandable regions outside of $B_{l_*}(0,\rho+\rho')$ towards the origin. 
In order to compute a lower bound on the time of the spread of these $\bar{\theta}$-affected nodes, we proceed as follows.
We observe that there are at most $\rho''$ affected nodes in $B_{l_*}(0,\rho+\rho')$ w.h.p, therefore if we remove a set of annuli centered at the origin from this ball such that each annulus contains at least one affected node of any type along with the possible clusters of bad blocks corresponding to that affected node and also a margin of good blocks around these clusters, we would have a ball of radius at least $\rho+\rho'-(\size^4+\size)\rho''$. Furthermore, we argue that since the event of having larger spreads of $\bar{\theta}$-affected nodes in a given time interval and the event of having this ball being a region of expansion of type $\bar{\theta}$ are both increasing in the change of a $\bar{\theta}$-particle to a $\theta$-particle, by assuming that this ball is a region of expansion of type  $\bar{\theta}$ we would   get a lower bound on the time  of the spread of the $\bar{\theta}$-affected nodes. Hence, from this point forward, to consider the spread of $\bar{\theta}$-affected nodes towards the origin we consider the ball $B_{l_*}(0,\rho+\rho'-(\size^4+\size)\rho'')$ and assume it is a region of expansion and does not contain any affected nodes. 

 Let $A_7$ denote the event that the $l_*$ radius of the spread of the $\bar{\theta}$-affected nodes in a time interval of $\rho'^{1/3}$ -- that conditional on $A_i$, $i=1,2,\ldots,6$ is at most needed for the formation of the $\theta$-affected $w$-blocks in the first gradual spread (as described above) -- is less than  $\rho'^{1/3}+ \size^C\rho'^{1/6}\log^{3/2}\rho' $ in the $l_*$ norm in all directions in the annulus around the origin. To show that this event occurs w.h.p. we consider the $l_*$-ball $B_{l_*}(0,\rho+\rho'-(\size^4+\size)\rho'')$ described above. By coupling with our modified model, it follows from the proof of Proposition~\ref{Prop:Tess1} and Theorem~\ref{Thrm:Shape_theorem} that the $l_*$ radius of the spread of any set of $\bar{\theta}$ expandable region adjacent to this ball inside this ball will be less than $\rho'^{1/3}+ \size^C\rho'^{1/6}\log^{3/2}\rho'$, hence we can conclude that  
\begin{align*}
A_7  \mbox{ occurs w.h.p. }
\end{align*}

Now let $A_8$ denote the event that the origin is contained in a $w$-block of $\theta$-affected nodes before there are any $\bar{\theta}$-affected nodes in an $l_*$-ball with radius $\rho'/2$ around the origin. To show that this event occurs w.h.p. we consider $L$ time intervals of size $\rho'/4$ and argue that, first of all, in every one of these time intervals the spread of $\theta$-affected nodes from the closest $\theta$-affected $w$-block towards the origin (first started from the expandable region) is at least $\rho'/4-\size^C(\rho')^{1/2}\log^{3/2} \rho'$. 
To see this, consider a line segment of length $\rho$ and let  $\mathcal{N}'$ denote the set of particles such that their $l_*$-distances from at least one point on the line segment is less than or equal to $\size^{c'}\rho'$ where $c'$ is the constant $C_1$ in Lemma~\ref{Lemma:lemma1'}.  We argue that not having $\theta$-affected nodes in $\mathcal{N}'$, and having smaller spreads in a given time interval in this neighborhood are both increasing events in the change of a $\bar{\theta}$-particle to a $\theta$-particle hence using the FKG inequality they are positively correlated. 
Let us consider a neighborhood with the shape of $\mathcal{N}'$ that does not contain any affected nodes and is a region of expansion. By coupling with our modified model, it follows from the proof of Proposition~\ref{Prop:Tess1} and Theorem~\ref{Thrm:Shape_theorem} that the spread of an affected block in this neighborhood will be at least $\rho'/4-\size^C(\rho')^{1/2}\log^{3/2} \rho'$. 
It follows that all the spreads of $\theta$-affected nodes started by the expandable region towards the origin during each interval will be at least $\rho'/4-\size^C(\rho')^{1/2}\log^{3/2} \rho'$. 

We now show that the possible spreads of $\bar{\theta}$-affected nodes started from outside of $B_{l_*}(0,\rho+\rho'-w)$ are not going to reach or in any way interfere with any of these spreads of $\bar{\theta}$-affected nodes and will not reach the $B_{l_*}(0,\rho'/2)$ before having the origin contained in a $\theta$-affected $w$-block (see Figure~\ref{fig:gradual_growth2}). To show this, since we are conditioning on events $A_i$, $i=1,2,\ldots,7$, and since $LL' <  \rho'/4-\rho'^{1/3}-(\size^4+\size)\rho''-\size$ it suffices to show that the spreads of $\bar{\theta}$-affected nodes started from outside $B_{l_*}(0,\rho+\rho'-(\size^4+\size)\rho''-\rho'^{1/3}-\size^c\rho''-\size)$ which does not contain any affected nodes and is assumed to be a region of expansion in every time interval of size $\rho'/4$ is at most $\rho'/4+\size^C(\rho')^{1/2}\log^{3/2} \rho'$, which guarantees that not only these spreads will not interfere with the spreads of $\theta$-affected nodes but also they will not reach  $B_{l_*}(0,\rho'/2)$ before having the origin contained in a $\theta$-affected $w$-block. By coupling with our modified model, it follows from the proof of Proposition~\ref{Prop:Tess1} and Theorem~\ref{Thrm:Shape_theorem} that this event also occurs w.h.p. Hence we can conclude that 
 \begin{align*}
 A_8 \mbox{ occurs w.h.p.}
 \end{align*}

This also implies that by the time the origin is contained in a $\theta$-affected $w$-block, the $\bar{\theta}$-affected nodes are still in $l_*$-distance of more than $\rho'/2$ from the origin w.h.p. 


Now let $r$ be proportional to $\rho'/\size$. Let us denote the event that the time it takes until a number of affected nodes equal to the number of all the particles in a firewall with radius $r$ centered at the origin and a line of width $2\sqrt{\size}$ from the origin to the firewall make a flip one by one being smaller than $\rho'/4$ by $A'_9$  (see Figure~\ref{fig:firewall_formation}).  Standard concentration bounds imply $A'_9 \mbox{ occurs w.h.p.}$

Let $A_{9}$ denote the event that this firewall is formed in a time interval smaller than $\rho'/4$. We have $P(A_{9}) \ge P(A'_{9})$ and since $A'_{9}$ occurs w.h.p. we have
\begin{align*}
A_{9} \mbox{ occurs w.h.p.}
\end{align*} 

With a similar argument for event $A_8$, w.h.p. the spreads of all the possible $\bar{\theta}$-affected nodes will be smaller than $\rho'/3$ for this interval (event $A_{10}$). Hence, we have
\begin{align*}
A_{10} \mbox{ occurs w.h.p.}
\end{align*}

 Finally we can write
 \begin{align*}
 P(A) \ge P\left(A\given[\Big] A_0,A_1,\ldots,A_{10}\right)P(A_0\cap A_1 \cap \ldots \cap A_{10}),
 \end{align*}
hence, we have that 
\begin{align*}
A \mbox{ occurs w.h.p.}
\end{align*}
Now, using Lemma~\ref{Lemma:monoch_spread_1} w.h.p. the interior of the firewall is a region of expansion in the initial configuration and since w.h.p. only $\theta$-affected nodes have reached this region by the time of the formation of the firewall, it is still a region of expansion for the state $\bar{\theta}$. Now, since the sum of the times of the gradual spreads, formation of the firewall, and the time that it takes until the interior of the firewall becomes monochromatic (by a standard concentration bound) is less than $\ti^*$ w.h.p., for all $\ti\ge \ti^*$ it will be monochromatic w.h.p. and this proves the lower bound.

\begin{figure}[!t] 
\centering
\includegraphics[width=1.6in]{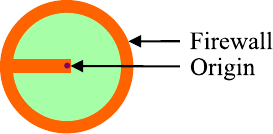}
\caption{Formation of a firewall around the origin.} 
\label{fig:firewall_formation}
\end{figure}


Next, we show the corresponding upper bound. 
Consider four neighborhoods with radius $\size(\rho+\rho')$ such that each of them shares the origin as a different corner node. Divide the union of these neighborhoods into neighborhoods of radius $\rho+\rho'$ in an arbitrary way and consider the nodes at the center of each of these neighborhoods. Now using the above result we have that for $\ti\ge \ti^*$, w.h.p., all these central nodes will have a monochromatic balls of size at least $2^{a(\tau)-\epsilon\size}$. Also it is easy to see that for $\ti\ge \ti^*$, w.h.p. all the four neighborhoods defined above will have particles with exponentially large monochromatic balls of both states. This implies that for all $\ti\ge \ti^*$ the size of the monochromatic ball located at the origin is at most $4\size^2(\rho+\rho')^2$. 

\end{proof}

\subsection{Extension to intolerance parameters larger than 1/2}\label{Subsec:Extension}
As mentioned before, while for $\tau < 1/2$ unstable particles are also p-stable, for $\tau > 1/2$ this is not the case.

Let $\bar{\tau} = 1 - \tau + 2/{\size}$. A p-stable particle of type $\theta$ is a particle for which $W < \bar{\tau}{\size}$ where $W$ is the number of $\theta$ particles in its neighborhood. The reason for adding the term $2/{\size}$ in the definition is to account for the strict inequality that is needed for being p-stable and the flip of the particle at the center of the neighborhood which adds one particle of its type to the neighborhood. 
All our results can be easily extended for $\tau > 1/2$ using $\bar{\tau}$ defined above. 
For example, a radical region in this case is a neighborhood $\mathcal{N}_S$ of radius $S=(1+\epsilon')w$ such that $W_S <  \bar{\tau}'(1+\epsilon')^2{\size} $, where    $\epsilon \in (0,1/2)$ and
\begin{equation*} %
\bar{\tau}' = \left(1-\frac{1}{\bar{\tau} {\size}^{1/2-\epsilon}}\right)\bar{\tau}. 
\end{equation*}

By replacing $\tau$ with $\bar{\tau}$,  it can be checked that all proofs extend to the interval $1/2<\tau<1-\tau_*$ for the shape theorem and the interval $1/2<\tau<1-\tau^*$ for the size theorem.

\section{Conclusions and Future Directions} \label{Sec:Conclusion}
We provide two key theorems for a spin system located on a flat torus or on  $\mathbb{Z}^2$. While most of the previous theoretical developments for this process focus on its final or limiting configuration, our shape theorem provides a first-order characterization of  the geometry of ``affected nodes'' at any time during the evolution phase. Our second theorem provides the first result for the size of the largest monochromatic ball of any node in the final configuration, for a given interval of the intolerance parameter.
Along the way, we also  provided a  tight concentration bound for the spreading time of the affected nodes. We expect that the interval  of $\tau$ leading to exponential monochromatic balls can be further improved using the techniques developed in this paper,  using more complex geometric constructions. Also, we only discuss the existence of an $l_*$ norm for the shape theorem while simulations suggests an Euclidean norm and this remains to be proven.
Another direction of further study could be the investigation of how the parameter of the initial distribution of the agents influences the formation of monochromatic balls,  since it is only known that  a single monochromatic ball occurs w.h.p.  for $\tau=1/2$ and $p \in (1-\epsilon,1)$, while our results are  limited to  the case $p=1/2$. Finally, we point out that for $\tau = 1/2$ the behavior of the model is unknown.

\section*{Acknowledgment}
The authors would like to thank Prof. Jason Schweinsberg (Math Department, UC   San Diego) for providing invaluable feedback and detailed suggestions on earlier drafts of the paper.

\bibliography{Refs}

\end{document}